\documentclass[11pt]{article}

\usepackage{url}
\usepackage{lineno}
\usepackage{calc}
\usepackage[compact]{titlesec}
\usepackage{caption}
\usepackage[paper=letterpaper,
hmargin={1.2in,1.2in},
vmargin={1in,1in},
]{geometry}
\usepackage{fancyhdr}
\usepackage{amsmath}
\usepackage{amsthm}
\usepackage{amssymb}
\usepackage{mathtools}
\usepackage{relsize}
\usepackage{graphicx}
\usepackage{color}
\usepackage{subfig}
\usepackage{tabu}
\usepackage{tikz}
\usetikzlibrary{decorations.pathmorphing,shapes}
\usepackage[shortlabels]{enumitem}
\usepackage{verbatim}
\usepackage{tikzsymbols}
\usepackage{enumitem}
\usepackage{hyperref}
\usepackage{tikz-cd}
\usepackage{pgfplots}
\usepackage{thmtools}
\usepackage{standalone}
\usepackage{algorithm2e}
\usepackage{xspace}
\usepackage{booktabs}
\usepackage{multirow}				 

\theoremstyle{plain}
\newtheorem{theorem}{Theorem}
\newtheorem{corollary}[theorem]{Corollary}
\newtheorem{lemma}[theorem]{Lemma}
\newtheorem{proposition}[theorem]{Proposition}

\theoremstyle{definition}
\newtheorem{definition}[theorem]{Definition}

\theoremstyle{remark}

\newcommand{\mytextsymbol}[1]{\ensuremath{\mathsf{#1}}\xspace}

\newcommand{\OBCe}{\BCe{1}}
\newcommand{\KBCv}{\BCv{k}}
\newcommand{\KBCe}{\BCe{k}}
\newcommand{\BCv}[1]{${#1}$-\mytextsymbol{BC}}
\newcommand{\BCe}[1]{${#1}$-$\mathsf{BC}^e$\xspace}
\newcommand{\BRe}[1]{${#1}$-$\mathsf{BR}^e$\xspace}
\newcommand{\BRv}[1]{${#1}$-\mytextsymbol{BR}}
\newcommand{\BVv}[1]{${#1}$-\mytextsymbol{BV}}
\newcommand{\BVe}[1]{${#1}$-$\mathsf{BV}^e$\xspace}
\newcommand{\OBRv}{\BRv{1}}

\newcommand{\NP}{\mytextsymbol{NP}}
\renewcommand{\P}{\mytextsymbol{P}}
\newcommand{\GI}{\mytextsymbol{GI}}
\newcommand{\weight}{\mytextsymbol{weight}}
\newcommand{\Hit}{\mytextsymbol{HittingSet}}
\newcommand{\hit}{\text{hit}}
\newcommand{\phiv}{\Phi}
\newcommand{\phiw}{\Phi}
\newcommand{\phic}{\Phi_{\mathsf{c}}}
\newcommand{\cA}{\mathcal{A}}
\newcommand{\cB}{\mathcal{B}}
\newcommand{\cF}{\mathcal{F}}

\newcommand{\cS}{\mathcal{S}}
\newcommand{\Z}{\mathbb{Z}}
\newcommand{\SAP}{\mytextsymbol{SAP}}
\newcommand{\tr}{\text{tr}}
\newcommand{\N}{\mathbb{N}}

\title{Computational Complexity of $k$-Block Conjugacy}
\author{Tyler Schrock \qquad Rafael Frongillo\\[2pt]University of Colorado, Boulder}

\begin{document}

\maketitle

\begin{abstract}
  We consider several computational problems related to conjugacy between subshifts of finite type, restricted to $k$-block codes: verifying a proposed $k$-block conjugacy, deciding if two shifts admit a $k$-block conjugacy, and reducing the representation size of a shift via a $k$-block conjugacy.
  We give a polynomial-time algorithm for verification, and show \GI- and \NP-hardness for deciding conjugacy and reducing representation size, respectively.
  Our approach focuses on $1$-block conjugacies between vertex shifts, from which we generalize to $k$-block conjugacies and to edge shifts.
  We conclude with several open problems.
\end{abstract}

\section{Introduction}

One-dimensional subshifts of finite type (SFTs) are of fundamental importance in the study of symbolic dynamical systems.
Despite their central role in symbolic dynamics, however, several basic questions about SFTs remain open, particularly with regard to computation.
Most prominent is the conjugacy problem: whether it is possible to decide if two given SFTs are conjugate.
In this work, we study restricted versions of the conjugacy problem, with an eye toward applications (algorithms to simplify representations of SFTs) as well as developing insights toward the full conjugacy problem.
In particular, we address the computational complexity of deciding or verifying conjugacy when given a bound on the block size of the corresponding sliding block code.
We focus on the case of vertex shifts; see below for other representations, notably edge shifts.

First consider the question of verification: given two vertex shifts and a proposed sliding block code, what is the computational complexity of verifying that the code induces a conjugacy?
We give a polynomial-time algorithm, for both the irreducible and reducible cases; a polynomial-time algorithm in latter case was not known (\S~\ref{sec:verification}).
Second, the question of deciding $k$-block conjugacy: given two vertex shifts, what is the complexity of deciding if there exists a sliding block code, with block length at most $k$, that induces a conjugacy?
By the first result on efficient verification, this problem is in \NP; we show it to be  \GI-hard (at least as hard as the Graph Isomorphism problem) for all $k$ (\S~\ref{sec:conjugacy}).
Third, the question of reduction: given a vertex shift and integer $\ell$, what is the complexity of deciding whether there exists a $k$-block conjugacy which reduces the number of vertices by $\ell$?
Extending a construction from previous work~\cite{SAP}, we show that this problem, for $k=1$, is \NP-complete (\S~\ref{sec:reducing}).

It is interesting to contrast our results with those of previous work~\cite{SAP}, on the special case of $k=1$ with the restriction that the block code be a sequence of amalgamations.
(Recall that any conjugacy can be expressed as a sequences of splittings followed by amalgamations; see \S~\ref{sec:setting}.)
This previous work shows that the analogous version of our third problem, of reducing the number of vertices using only amalgamations, is \NP-complete, but it does not address the verification problem; intuitively it seems plausible that verification would also be \NP-hard.
Returning to our setting, note that general 1-block codes need not be sequences of amalgamations (Figure~\ref{fig:a_dot_map}).
Thus, while it is unsurprising that the reduction problem remains \NP-hard in our setting, it is perhaps surprising given that verification can be done in polynomial time, as a priori the number of splittings required could be super-polynomial.

Edge shifts have received more attention in the literature, perhaps because of their succinct representations as integer matrices.
Precisely because of their succinct representations, the question of verification is somewhat nuanced: verifying a given sliding block code requires writing down the proposed code, which can be exponential in the description size of the original edge shifts, so while the runtime of our algorithm can be exponential in the description sizes of the shifts, it is still polynomial-time (\S~\ref{sec:edge-shifts}).
We also show \GI-hardness for the corresponding conjugacy problems, and leave several open questions (\S~\ref{sec:discussion}).

\section{Setting}
\label{sec:setting}

We begin with basic graph-theoretic definitions and convention.
A directed graph $G = (V,E)$ is a set of vertices $V$ along with a set of edges $E\subseteq V\times V$.
When multiple graphs are in play, we will write $G=(V_G,E_G)$ to clarify which graphs the vertices or edges correspond to.
For a directed graph $G=(V,E)$ and a vertex $v\in V$, we define $N^+(v)=\{u\in V: (v,u)\in E\}$ and $N^-(v)=\{u\in V: (u,v)\in E\}$ to be the set of out-neighbors and in-neighbors of $v$, respectively.
Unless specified otherwise, \emph{a cycle of length $n$} will mean a sequence $v_1v_2\cdots v_n \in V$ such that $(v_i,v_{i+1})\in E$ for $i\in\{1,\ldots,n+1\}$ where $v_{n+1} := v_1$.
That is, $v_1v_2\cdots v_n$ is a cycle in our terminology if the path $v_1v_2\cdots v_nv_1$ forms a cycle in $G$.
We define $C_n(G)$ to be the set of cycles of length $n$ in $G=(V,E)$.
For example, if the path $v_1v_2v_3v_1$ forms a cycle in $G$, then $v_1v_2v_3, v_2v_3v_1, v_3v_1v_2 \in C_3(G)$.

Let $\cA$ be a finite set.
The \emph{full shift} $\cA^\Z$ over alphabet $\cA$ is the set $\{(x_i)_{i\in \Z}:x_i\in\cA\text{ for all }i\in\Z\}$.
An element of $\cA^\Z$ is called a \emph{point}.
A \emph{block} (or \emph{word}) in $\cA$ is a string $a_1a_2\cdots a_n$ of symbols from $\cA$.
We will use the term \emph{infinite word} to describe strings in $\cA$ which are infinite in exactly one direction.
If $x=(x_i)_{i\in \Z}\in \cA^\Z$, we use $x_{[a,b]}$ to denote the block $x_ax_{a+1}\cdots x_b$.
Similarly, we use $x_{[a,\infty)}$ to denote the infinite word $x_ax_{a+1}\cdots$.
Let $\cF$ be a set of blocks over $\cA$ called \emph{forbidden blocks}.
Then $X_\cF$ is defined to be the subset of $\cA^\Z$ where each $x\in X_\cF$ contains none of the forbidden block in $\cF$.
A \emph{shift space} (or \emph{shift}) is a subset $X\subseteq \cA^\Z$ such that $X=X_\cF$ for some set of forbidden blocks $\cF$.
If there exists a finite set $\cF$ such that $X=X_\cF$, then $X$ is called a \emph{shift of finite type}.

Given a directed graph $G=(V,E)$ with labeled vertices (each distinct), we associate to it the shift space $X_G=\{(v_i)_{i\in\Z}:v_i\in V,(v_i,v_{i+1})\in E\text{ for all }i\in\Z\}$, which is the collection of all bi-infinite walks on $G$.
Note that $X_G$ is a shift of finite type with $\cF=\{v_iv_j:(v_i,v_j)\notin V\}$.
Any shift space of this form is called a \emph{vertex shift}.
Similarly, given a directed multigraph $G=(V,E)$, i.e.\ where $E$ is a multiset, and a labeling of the edges from $\cA$, we define the \emph{edge shift} $X_G^e$ of labelings of bi-infinite walks on $G$.
Again edge shifts are shifts of finite type with $\cF=\{e_1e_2:e_1\text{ does not terminate at the initial vertex of }e_2\}$.

A shift $X$ is \emph{irreducible} if for every pair of words $w_1,w_2$ in $X$, there is a word $w_3$ such that $w_1w_3w_2$ is a word in $X$, and $X$ is \emph{reducible} if it is not irreducible.
In graph-theoretic terms, first consider any graph containing a vertex with either no out-neighbors or no in-neighbors.
Such a vertex is called \emph{stranded}.
A graph (or multigraph) with no stranded vertices is called \emph{essential}.
A graph (or multigraph) with the property that for every pair of vertices $u,v$ there is a path from $u$ to $v$ is called \emph{strongly connected}.
Finally, a vertex shift $X_G$ (or edge shift $X^e_G$) is irreducible if $G$ is essential and strongly connected.

Given a shift $X$ with alphabet $\cA_1$, we can transform $X$ into a shift space over another alphabet $\cA_2$ in the following way.
Fix integers $m,a$ with $-m\leq a$. Then letting $\cB_n(X)$ denote the set of blocks of size $n$ from the shift $X$ and given a function $\Phi:\cB_{m+a+1}(X)\to \cA_2$, the corresponding \emph{sliding block code} with memory $m$ and anticipation $a$ is the function $\Phi_\infty$ defined by $\Phi_\infty((x_i)_{i\in\Z})=(\Phi(x_{[i-m,i+a]}))_{i\in\Z}$.
That is, $\Phi_\infty$ looks at a block of size $m+a+1$ through a window to determine a character from $\cA_2$.
Then the window is slid infinitely in both directions.
Letting $k=m+a+1$, we will call any sliding block code with window size $k$ a \emph{$k$-block code}.
Given a sliding block code as $\Phi:\cA_1^k\to \cA_2$, we extend $\Phi$ to all finite and infinite words $w$ of length at least $k$ by $\Phi((w_i)_{i\in I})=(\Phi(w_{[i-m,i+a]}))_{i-m,i+a\in I}$, where $I \subsetneq \Z$.
That is, we extend $\Phi$ to words by sliding $\Phi$ over the entire word.

Let $X$ be any shift space with alphabet $\cA_1$.
We define the $k$th higher block shift $X^{[k]}$ with alphabet $\cA_2=\cB_k(X)$ by the image of $X$ under $\beta_N:X\to (\cB_k(X))^\Z$ where for any point $p\in X$, $\beta_N(p)_i=p_{[i,i+k-1]}$.
If $X=X_G$ happens to be a vertex shift, we can construct the $k$th higher block shift in terms of the graph.
For any directed graph $G$, construct the graph $G^{[k]}$ by $V_{G^{[k]}}=\{v_1\cdots v_k:v_1\cdots v_k\text{ is a path in $G$}\}$ and $E_{G^{[k]}}=\{(v_1v_2\cdots v_k, v_2\cdots v_kv_{k+1}):v_1\cdots v_{k+1}\text{ is a path in $G$} \}$.
Then $X_{G^{[k]}}=X_G^{[k]}$.
When dealing with $k$-block codes, it is often useful to \emph{pass to a higher block shift} by noting that there is a $k$-block conjugacy $\Phi_\infty:X\to Y$ if and only if there is a 1-block conjugacy $\Phi^{[k]}_\infty:X^{[k]}\to Y$~\cite[Proposition 1.5.12]{lind1999introduction}.

Furthermore, any sliding block code $\Phi_\infty:X_G\to X_H$ between vertex shifts induces the function $\phic:\bigcup_{n=1}^\infty C_n(G)\to \bigcup_{n=1}^\infty C_n(H)$ defined as follows.
Given a cycle $c$ in $G$, there is a unique cycle $d$ in $H$ with $|c|=|d|$ such that $\Phi_\infty(c^\infty)=d^\infty$; we set $\phic(c)=d$.
In the special case of a 1-block code, the block map $\Phi:\cA_1\to\cA_2$ is simply a map between the alphabets.
In this case, when $X_G$ is a vertex shift, we have $\phic(v_1\cdots v_n)=\Phi(v_1)\cdots\Phi(v_n)$.

\begin{definition}
  Let $X_G$ be a vertex shift.
  We say states $u,v\in V_G$ can be \emph{amalgamated} if one the following conditions is met.
  \begin{enumerate}
  \item $N^+(u)=N^+(v)$ and $N^-(u)\cap N^-(v)=\emptyset$
  \item $N^-(u)=N^-(v)$ and $N^+(u)\cap N^+(v)=\emptyset$
  \end{enumerate}
  If $u$ and $v$ are amalgamated, they are replaced by the vertex $uv$ which has $N^+(uv)=N^+(u)\cup N^+(v)$ and $N^-(uv)=N^-(u)\cup N^-(v)$.
\end{definition}

\begin{definition}
  Let $X_G$ be a vertex shift.
  A vertex $v\in V_G$ can be split into two vertices $v_1$ and $v_2$ provided the edges of $v_1,v_2$ satisfy one of the following conditions.
  \begin{enumerate}
  \item $\{N^+(v_1),N^+(v_2)\}$ is a partition of $N^+(v)$ and $N^-(v_1)=N^-(v_2)=N^-(v)$.
  \item $\{N^-(v_1),N^-(v_2)\}$ is a partition of $N^-(v)$ and $N^+(v_1)=N^+(v_2)=N^+(v)$.
  \end{enumerate}
  The corresponding new graph is called a \emph{state splitting} of $v$.
  Note that state splittings and amalgamations are inverse operations.
\end{definition}

The definitions for edge shifts are similar.
Since edges shifts are based on multigraphs, $N^-(v)$ and $N^+(v)$ are multisets.
The definition of a state splitting is identical noting that the partition is a multiset partition.
For amalgamations, two vertices $u,v$ can be amalgamated if $N^-(u)=N^-(v)$ or $N^+(u)=N^+(v)$.
In the case where $N^-(u)=N^-(v)$, $u,v$ are replaced by a single vertex $uv$ with $N^-(uv)=N^-(u)=N^-(v)$ and $N^+(uv)=N^+(u)\uplus N^+(v)$, where $N^+(u)\uplus N^+(v)$ is the multiset disjoint union.

\begin{theorem}[\cite{williams1973classification,lind1999introduction}]\label{thm:splitting->amalgamations}
  Let $X_G,X_H$ be vertex shifts (or edge shifts).
  Then $X_G$ and $X_H$ are conjugate if and only if there is a sequence of state splittings followed by a sequence of amalgamations which transform $G$ into $H$.
\end{theorem}

In the case of a 1-block code $\Phi:V_G\to V_H$, we may view the block map as a partition of the vertices of $G$, where each element of the partition is converted to a vertex of $H$.
In light of Theorem~\ref{thm:splitting->amalgamations}, it may be tempting to think that every 1-block code can be written as a sequence of amalgamations only, as intuitively splitting a vertex while requiring the vertices be re-amalgamated has no benefit.
Yet this statement is not true; there are simple examples of two graphs admitting a 1-block conjugacy, where no pair of vertices can be amalgamated in either graph (Figure~\ref{fig:a_dot_map}).

\begin{figure}
  \centering
  \begin{tabular}{cc}
    \raisebox{40pt}{(a)} & 
    \begin{tikzpicture}[scale=.8]
    \node (0) at (0,0) {$a$};
    \node (1) at (0,-2) {$b$};
    \node (2) at (-1,-1) {$c$};
    \node (3) at (-2,-2) {$d$};
    \node (4) at (-2,0) {$e$};
    \draw [->] (0) edge[bend right] (1) (1) edge[bend right] (0) (0) edge (2) (2) edge (3) (3) edge (4) (4) edge (0) (2) edge (1);
    \draw [->] (4) edge[loop above] (4);
    
    \node (la) at (1,-1){};
    \node (ra) at (4,-1){};
    \draw [->] (la) edge node[above] {$\begin{array}{c}a\mapsto a \\ b,c,d,e\mapsto b\end{array}$} (ra);
    
    \node (0b) at (5,0) {$a$};
    \node (1b) at (5,-2) {$b$};
    \draw [->] (0b) edge[bend right] (1b) (1b) edge[bend right] (0b) (1b) edge[loop below] (1b);
    \end{tikzpicture}
          \\[10pt]
	\raisebox{0pt}{(b)} &
  
  \begin{tabular}{c}
  \begin{tikzpicture}[scale=.8]
  \node (0) at (0,0) {$a$};
  \node (1) at (0,-2) {$b$};
  \node (2) at (-1,-1) {$c$};
  \node (3) at (-2,-2) {$d$};
  \node (4) at (-2,0) {$e$};
  \draw [->] (0) edge[bend right] (1) (1) edge[bend right] (0) (0) edge (2) (2) edge (3) (3) edge (4) (4) edge (0) (2) edge (1);
  \draw [->] (4) edge[loop above] (4);
  
  \node (la) at (1,-1){};
  \node (ra) at (2,-1){};
  \draw [->] (la) edge (ra);
  
  \node (0c) at (5,0) {$a$};
  \node (1ac) at (5,-2) {$b_1$};
  \node (1bc) at (6,-1) {$b_2$};
  \node (2c) at (4,-1) {$c$};
  \node (3c) at (3,-2) {$d$};
  \node (4c) at (3,0) {$e$};
  \draw [->] (0c) edge[bend right] (1bc) (1bc) edge[bend right] (0c) (1ac) edge (0c) (0c) edge (2c) (2c) edge (3c) (3c) edge (4c) (4c) edge (0c) (2c) edge (1ac);
  \draw [->] (4c) edge[loop above] (4c);
  
  \node (la2) at (7,-1){};
  \node (ra2) at (8,-1){};
  \draw [->] (la2) edge (ra2);
  
  \node (0d) at (11,0) {$a$};
  \node (2d) at (10,-1) {$c$};
  \node (1a3d) at (9,-2) {$b_1 d$};
  \node (4d) at (9,0) {$e$};
  \node (1bd) at (12,-1) {$b_2$};
  \draw[->] (0d) edge (2d) (2d) edge (1a3d) (1a3d) edge (4d) (4d) edge (0d) (4d) edge[loop above] (4d) (1a3d) edge[bend right] (0d) (0d) edge[bend right] (1bd) (1bd) edge[bend right] (0d);
\end{tikzpicture} \\
\begin{tikzpicture}[scale=.8]
\node (la3) at (13,-1) {};
\node (ra3) at (14,-1) {};
\draw[->] (la3) edge (ra3);

\node (0e) at (17,0) {$a$};
\node (1be) at (18,-2) {$b_2$};
\node (2e) at (16,-2) {$c$};
\node (1a34e) at (15,0) {$b_1de$};
\draw[->] (0e) edge (2e) (2e) edge (1a34e) (1a34e) edge (0e) (1a34e) edge[loop above] (1a34e) (0e) edge[bend right] (1be) (1be) edge[bend right] (0e);

\node (la4) at (19,-1) {};
\node (ra4) at (20,-1) {};
\draw[->] (la4) edge (ra4);

\node (0f) at (23,0) {$a$};
\node (1b2f) at (23,-2) {$b_2c$};
\node (1a34f) at (21,-1) {$b_1de$};
\draw[->] (0f) edge[bend right] (1b2f) (1b2f) edge[bend right] (0f) (1b2f) edge (1a34f) (1a34f) edge (0f) (1a34f) edge[loop above] (1a34f);

\node (la5) at (24,-1) {};
\node (ra5) at (25,-1) {};
\draw[->] (la5) edge (ra5);

\node (0b) at (26,0) {$a$};
\node (1b) at (26,-2) {$bcde$};
\draw [->] (0b) edge[bend right] (1b) (1b) edge[bend right] (0b) (1b) edge[loop below] (1b);
\end{tikzpicture}
\end{tabular}
  \end{tabular}
	\caption{(a) A minimal example of two vertex shifts which are conjugate by a 1-block code but by not a sequence of amalgamations. (b) The conjugacy, demonstrated via a splitting followed by four amalgamations.}
	\label{fig:a_dot_map}
\end{figure}
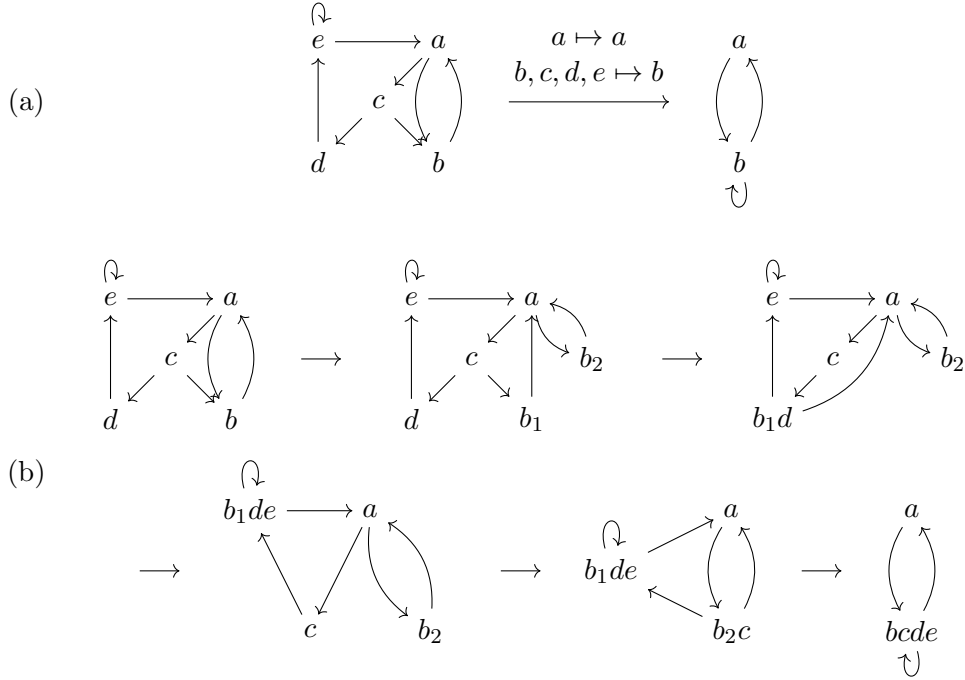

We conclude the background with a common way a sliding block code can fail to be injective.
Given a $k$-block code $\Phi_\infty:X_G\to X_H$, if there exist distinct words $w_2,w_2'$ such that $\Phi(w_1w_2w_3)=\Phi(w_1w_2'w_3)$ with $|w_1|=|w_3|=k$, we say $\Phi_\infty$ \emph{collapses a diamond}.
As we now state, if a sliding block code is injective, it cannot collapse a diamond.
(As we discuss in \S~\ref{sec:verify-reducible}, if $\phic$ is injective, collapsing a diamond is actually the only way $\Phi_\infty$ can fail to be injective.)
We prove the result for completeness; see, e.g.,~\cite[Theorem 8.1.16]{lind1999introduction} for a similar result in the irreducible case.

\begin{lemma}\label{lemma:diamond_lemma}
  Let $\Phi_\infty:X_G\to X_H$ be a $k$-block code.
  If $\phiw$ collapses a diamond, then $\Phi_\infty$ is not injective.
\end{lemma}
\begin{proof}
  Suppose $\Phi$ collapses a diamond.
  That is, $\Phi(w_1w_2w_3)=\Phi(w_1w_2'w_3)$ for some words $w_1,w_3$ of length $k$ and distinct words $w_2,w_2'$ in $G$.
  Consider any infinite word $w_0$ which can precede $w_1$ and any infinite word $w_4$ which can follow $w_3$.
  Then $\Phi_\infty(w_0w_1w_2w_3w_4)=\Phi_\infty(w_0w_1w_2'w_3w_4)$, so $\Phi_\infty$ is not injective.
\end{proof}

\section{Verification: Testing a $k$-Block Map for Conjugacy}
\label{sec:verification}

Given a pair of directed graphs $G,H$, and a proposed $k$-block map $\Phi$, we wish to verify whether or not $\Phi$ induces a conjugacy between the vertex shifts $X_G,X_H$.
We will focus in this section on the case $k=1$, as the case $k>1$ follows immediately by recoding to the $k$th higher block shift.
When $G$ and $H$ are irreducible (strongly connected), this problem boils down to checking that the two graphs have the same number of cycles of each length up to some constant, and furthermore that $\Phi$ induces an injection on these cycles.
While cycle counting can be done efficiently using powers of the adjacency matrices, the challenge remains of checking injectivity efficiently.

The reducible case, when $G$ and $H$ are not strongly connected, is much more complex.
We give counter-examples to several statements which would have led to a straightforward algorithm wherein one subdivides the graphs into their irreducible components and uses the algorithm for the irreducible case on each, together with some other global checks.
Instead, we give a more direct reduction to the irreducible case: we efficiently augment the graphs and block map with new vertices and edges, until the resulting graphs are irreducible, in such a way as to preserve conjugacy (or lack thereof).

\subsection{Irreducible Case}\label{sec:verify-irreducible}

As described above, we will focus first on 1-block codes.
When $G,H$ are irreducible,
the following straightforward topological result allows us to
restrict attention to the map induced on cycles between the graphs.

\begin{proposition}\label{prop:dense_subset}
  Suppose $X,Y$ are compact metric spaces, $\psi: X\to Y$ is continuous, and $D\subseteq Y$ is a dense subset of $Y$.
  If $\psi$ surjects onto $D$, then $\psi$ surjects onto all of $Y$.
\end{proposition}
\begin{proof}
  Suppose $X,Y$ are compact metric spaces, $\psi: X\to Y$ is continuous, $D\subseteq Y$ is a dense subset of $Y$, and $D\subseteq\psi(X)\subseteq Y$.
  Let $p\in Y$.
  Since $D$ is dense, there is a sequence $\{p_n\}$ in $D$ which converges to $p$.
  Since $\psi$ surjects onto $D$, every $p_n$ has a preimage in $X$.
  Pick $\gamma :D\to X$ such that $\psi\circ\gamma=\text{id}_D$.
  Then $\{\gamma(p_n)\}$ is a sequence in $X$.
  Since $X$ is compact, there is a subsequence $\{\gamma(p_{n_k})\}$ which converges to some $q\in X$.
  As limits commute with continuous functions, we have
  \[\psi(q) = \psi(\lim\limits_{n_k\to \infty} \gamma(p_{n_k}))=\lim\limits_{n_k\to \infty}\psi(\gamma(p_{n_k}))=\lim\limits_{n_k\to\infty}p_{n_k}=p.\]
  Thus $\psi(q)=p$, so $\psi$ is surjective on all of $Y$.
\end{proof}

We will applying Proposition~\ref{prop:dense_subset} with $D$ being the set of periodic points of $X_H$, which are in bijection with cycles of $H$.
The following result, that $\Phi$ induces a 1-block conjugacy if and only if it induces a bijection on cycles, appears to be known; we give the proof for completeness.

\begin{theorem}\label{thm:conjugacy_iff_cycle_bijection}
  Irreducible vertex shifts $X_G,X_H$ are conjugate via a 1-block code if and only if there is a vertex map $\phiv:V_G\to V_H$ such that the induced map $\phic$ is a bijection.
\end{theorem}
\begin{proof}
  If $\Phi_\infty$ is a conjugacy, then it is a bijection on periodic points; we conclude $\phic$ is a bijection.
  For the converse, suppose $\Phi_\infty$ is not a conjugacy.
  We proceed in cases.

  (Case 1) If $\Phi_\infty$ is not injective, there exist distinct points $p,q\in X_G$ such that $\Phi_\infty(p)=\Phi_\infty(q)$.
  (Case 1a) Suppose first that $p,q$ disagree at $|V_G|^2+1$ consecutive indices, meaning the words $p_{[a,b]},q_{[a,b]}$ disagree at every index for $a,b\in\Z$ with $b-a=|V_G|^2+1$.
  Consider all possible pairs of states in $G$; there are $|V_G|^2$ such pairs.
  Thus there exist distinct indices $c,d\in \{a,a+1\ldots,b\}$ such that $(p_c,q_c)=(p_d,q_d)$.
  But then $\phic(p_{[c,d-1]})=\phic(q_{[c,d-1]})$.

  (Case 1b) Suppose instead that $p,q$ do not disagree at $|V_G|^2+1$ consecutive indices: there exist indices $a,b$ with $a<b-1$ such that $p,q$ agree at indices $a$ and $b$, but $p,q$ disagree at every index between $a$ and $b$.
  Let $w$ be any word connecting $p_b=q_b$ to $p_a=q_a$.
  Then $\phic(p_{[a,b]}w)=\phic(q_{[a,b]}w)$.

  (Case 2) If $\Phi_\infty$ is not surjective, suppose for a contradiction that $\phic$ is bijective.
  Then every periodic point in $X_H$ is mapped to by $\Phi_\infty$.
  Since the periodic points are a dense subset of the compact metric space $X_H$, Proposition~\ref{prop:dense_subset} contradicts the fact that $\Phi_\infty$ is not surjective.
\end{proof}

To verify that the cycle map $\phic$ is bijective, we will test for injectivity explicitly, and rely on counting arguments to check surjectivity.
For injectivity, it turns out that checking cycles up to length $|V_G|^2$ suffices.

\begin{proposition}\label{prop:injective_is_decidible}
  Suppose $\Phi_\infty: X_G\to X_H$ is a 1-block code between irreducible vertex shifts.
  If $\phic$ is injective on $\bigcup_{n=1}^{|V_G|^2} C_n(G)$, then $\phic$ is injective.
\end{proposition}
\begin{proof}
  Let $c,d$ be distinct cycles of size $|c|=|d|=k>|V_G|^2$.
  Proceeding by strong induction, suppose $\phic$ is injective on all cycles of size less than $k$.
  There are $|V_G|^2$ possible pairs of states in $G$.
  Thus there exist distinct indices $a,b$ such that $(c_a,d_a)=(c_b,d_b)$.
  That is, $c_{[a,b-1]},d_{[a,b-1]}$ are cycles of the same length and $c_{[b,a-1]},d_{[b,a-1]}$ are cycles of the same length.
  Since $c,d$ were distinct, we can assume without loss of generality that $c_{[a,b-1]},d_{[a,b-1]}$ are distinct.
  By the induction hypothesis, $\phic(c_{[a,b-1]})\neq \phic(d_{[a,b-1]})$.
  Thus $\phic(c)\neq \phic(d)$.
\end{proof}

Proposition~\ref{prop:injective_is_decidible} suggests the na\"ive algorithm of checking all cycles up to length $|V_G|^2$ to verify injectivity of $\phic$.
This algorithm is remarkably inefficient, however; letting $n=|V_G|$, there can be $\Omega(n^{n^2})$ cycles of length up to $n^2$, as is the case for the complete graph.
Fortunately, these checks can be performed much more efficiently, by rephrasing them as a search problem in a graph built from pairs of vertices in $G$.
This procedure is outlined in Algorithm~\ref{alg:phi_c_injective}.

\begin{theorem}\label{thm:verify1block}
  Let $X_G$ be a vertex shift and $A=\{1,2,\ldots,m\}$.
  Then any given map $\phiv:V_G\to A$ induces a map $\phic:\bigcup\limits_n C_n(G)\to \bigcup\limits_n A^n$.
  Deciding if $\phic$ is injective can be determined in $O(|V_G|^4)$ time.
\end{theorem}
\begin{proof}
  First we build the directed meta-graph $M=(V_M,E_M)$ where $V_M=\{(v_1,v_2):v_1,v_2\in V_G\}$ and $E_M=\{((v_1,v_2),(u_1,u_2)):\phiv(v_1)=\phiv(u_1),\phiv(v_2)=\phiv(u_2),(v_1,u_1)\in E_G,\text{ and }(v_2,u_2)\in E_G\}$.
  That is, $M$ is a graph on pairs of vertices from $G$, with an edge connecting pairs $P_1,P_2$ if and only if
  (i) there is a pair of (possibly non-distinct) edges in $G$ connecting the two vertices in $P_1$ to the vertices in $P_2$,
  and (ii) the induced map on words of length two (i.e., edges) maps the two edges together.
  $M$ can be constructed in $O(|V_G|^4)$ time.

  Given $M$, the map $\phic$ is injective if an only if there is no cycle in $M$ which passes through a vertex $(v_1,v_2)\in V_M$ with $v_1\neq v_2$.
  Furthermore, such a cycle in $M$ exists if and only if $M$ has a strongly connected component containing an edge and a vertex $(v_1,v_2)$ with $v_1\neq v_2$.
  Tarjan's strongly connected components algorithm~\cite{tarjan1972depth} now applies, in $O(|V_M|+|E_M|)=O(|V_G|^4)$ time.
\end{proof}

Putting the above results together with the higher-block codes gives the desired algorithm to verify $k$-block conjugacies; the full conjugacy algorithm for $k=1$ is outlined in Algorithm~\ref{alg:phi_c_bijective}.

\begin{corollary}\label{cor:verify_conjugacy_irreducible}
  Given a $k$-block code $\Phi_\infty:X_G^k\to X_H$ between irreducible vertex shifts, deciding if $\Phi_\infty$ is a conjugacy is in \P.
  In particular, it can be determined in $O(|V_G|^{4k})$ time.
\end{corollary}
\begin{proof}
  Given $G,H$, we first pass to the $k$th higher block shift $X_{G^{[k]}}$ of $X_G$, recalling that $\Phi_\infty^{[k]}$ is a 1-block code and $\Phi_\infty$ is a conjugacy if and only if $\Phi_\infty^{[k]}$ is a conjugacy~\cite[Proposition 1.5.12]{lind1999introduction}.
  We can construct $\Phi^{[k]}_\infty:X_{G^{[k]}}\to X_H$ in time $O(|V_{G^{[k]}}| + |E_{G^{[k]}}|) = O(|V_{G^{[k]}}|^3)$.
  Noting that $|V_{G^{[k]}}|\leq |V_G|^k$, it thus suffices to show the case $k=1$.

  By Theorem~\ref{thm:conjugacy_iff_cycle_bijection}, $\Phi_\infty$ is a conjugacy if and only if $\phic$ is a bijection.
  As $k=1$, Theorem~\ref{thm:verify1block} shows that injectivity of $\phic$ can be determined in $O(|V_G|^4)$ time.
  To show $\phic$ is surjective, it suffices to check that $|C_i(G)|=|C_i(H)|$ for all $i\in\N$.
  Letting $A(G),A(H)$ be the adjacency matrices of $G,H$, we note $|C_i(G)|=\tr(A(G)^i)$, so our desired check is equivalent to checking $\tr(A(G)^i)=\tr(A(H)^i)$ for all $i\in\N$~\cite[Proposition 2.2.12]{lind1999introduction}.
  In fact, it suffices to check up to $i=|V_{G}|$~\cite{hou1998leverrierfadeev,leverrier1840}.
  Calculating $\tr(A(G)^i),\tr(A(H)^i)$ for all $i\in\{1,\ldots,|V_G|\}$ can be done by repeated multiplication in $O(|V_G|^{1+\omega})=O(|V_G|^4)$ time, where $\omega$ is the exponent of matrix multiplication.
\end{proof}

\subsection{Reducible Case}\label{sec:verify-reducible}

Several useful statements about conjugacy between irreducible vertex shifts fail to hold in the reducible case.
First, given a sliding block code $\Phi_\infty:X_G\to X_H$ between irreducible vertex shifts, it is known that if $\Phi_\infty$ is injective and $G,H$ have the same topological entropy, then $\Phi_\infty$ is a conjugacy~\cite[Corollary 8.1.20]{lind1999introduction}.
(The topological entropy of a shift $X$ is defined as $h(X)=\lim\limits_{n\to\infty}\frac{1}{n}\log_2|\cB_n(X)|$.)
If the shifts are reducible, however, $\Phi_\infty$ can satisfy these conditions but fail to be surjective (Figure~\ref{fig:simple_reducible_examples}a).
Second, we have from Theorem~\ref{thm:conjugacy_iff_cycle_bijection} that if $\Phi_\infty$ is a 1-block code between irreducible vertex shifts, then $\phic$ being a bijection implies $\Phi_\infty$ is bijective.
In the reducible case, $\phic$ can be bijective while $\Phi_\infty$ fails injectivity (Figure~\ref{fig:simple_reducible_examples}b) or surjectivity (Figure~\ref{fig:simple_reducible_examples}a).

As an even stronger test, one might guess for reducible vertex shifts that if $\Phi_\infty:X_G\to X_H$ is surjective and the induced maps between irreducible subgraphs are all conjugacies, then $\Phi_\infty$ is a conjugacy.
If true, this statement would suggest applying the algorithm in Corollary~\ref{cor:verify_conjugacy_irreducible} to each irreducible subgraph, at which point one would only need to test surjectivity.
Yet this statement is also false; $\phic$ being a bijection implies neither the injectivity nor the surjectivity of $\Phi_\infty$ (Figure~\ref{fig:simple_reducible_examples}).
By extending the argument of Theorem~\ref{thm:conjugacy_iff_cycle_bijection}, one can correct the statement by adding a check for diamonds: if $\Phi_\infty$ is surjective, the induced maps between irreducible subgraphs are all conjugacies, and $\Phi$ does not collapse a diamond, then $\Phi_\infty$ is a conjugacy.
Unfortunately, while this revised statement does break the problem of verifying a proposed 1-block conjugacy into more manageable pieces, how to turn it into a decision procedure, let alone an efficient algorithm, is far from clear.

\begin{figure}
  \centering
  \begin{tabular}{cc}
  	\raisebox{55pt}{(a)} &
    \begin{tikzpicture}[scale=1]
    \node (a) at (0,0) {$a$};
    \node (b) at (1,0.5) {$b$};
    \node (d) at (1,-0.5) {$d$};
    \node (c) at (2,0.5) {$c$};
    \node (e) at (2,-0.5) {$e$};
    \node (f) at (3,0) {$f$};
    \draw [<-] (a) edge (b) (b) edge (c) (c) edge (f) (a) edge (d) (d) edge (e) (e) edge (f) (f) edge[out=90, in=90] (a);
    
    \node (g) at (1.5,-2) {$g$};
    \draw [->] (g) edge[loop below] (g);
    \draw[->] (d) edge (g);
    
    \node (la) at (3.75,0){};
    \node (ra) at (5.25,0){};
    \draw [->] (la) edge (ra);
    
    \node (a2) at (6,0) {$a$};
    \node (b2) at (7,0) {$bd$};
    \node (c2) at (8,0.5) {$c$};
    \node (e2) at (8,-0.5) {$e$};
    \node (f2) at (9,0) {$f$};
    \draw [<-] (a2) edge (b2) (b2) edge (c2) (c2) edge (f2) (b2) edge (e2) (e2) edge (f2) (f2) edge[out=90, in=90] (a2);
    
    \node (g2) at (7.5,-2) {$g$};
    \draw [->] (g2) edge[loop below] (g2);
    \draw[->] (b2) edge (g2);
    \end{tikzpicture}
  	\\[10pt]
  	\raisebox{55pt}{(b)} & 
    \begin{tikzpicture}[scale=1]
    \node (a) at (0,0) {$a$};
    \node (b) at (1,0.5) {$b$};
    \node (d) at (1,-0.5) {$d$};
    \node (c) at (2,0.5) {$c$};
    \node (e) at (2,-0.5) {$e$};
    \node (f) at (3,0) {$f$};
    \draw [<-] (a) edge (b) (b) edge (c) (c) edge (f) (a) edge (d) (d) edge (e) (e) edge (f) (f) edge[out=90, in=90] (a);
    
    \node (g) at (1.5,-2) {$g$};
    \draw [->] (g) edge[loop below] (g);
    \draw[->] (c) edge[out=240, in=100] (g) (e) edge (g);
    
    \node (la) at (3.75,0){};
    \node (ra) at (5.25,0){};
    \draw [->] (la) edge (ra);
    
    \node (a2) at (6,0) {$a$};
    \node (b2) at (7,0.5) {$b$};
    \node (d2) at (7,-0.5) {$d$};
    \node (c2) at (8,0) {$ce$};
    \node (f2) at (9,0) {$f$};
    \draw [<-] (a2) edge (b2) (b2) edge (c2) (c2) edge (f2) (a2) edge (d2) (d2) edge (c2) (f2) edge[out=90, in=90] (a2);
    
    \node (g2) at (7.5,-2) {$g$};
    \draw [->] (g2) edge[loop below] (g2);
    \draw[->] (c2) edge (g2);
    \end{tikzpicture}
  \end{tabular}
  \caption{
    Counter-examples showing various statements which hold in the irreducible case fail in the reducible case.
    Note that all four shifts have the same topological entropy, $h(X)=\frac{1}{4}$.
    (a) A 1-block code between two reducible shifts which restricts to conjugacies between the irreducible components (and hence $\phic$ is a bijection) but is not surjective.
    (b) A 1-block code between two reducible shifts which restricts to conjugacies between the irreducible components but is not injective.}
  \label{fig:simple_reducible_examples}
\end{figure}
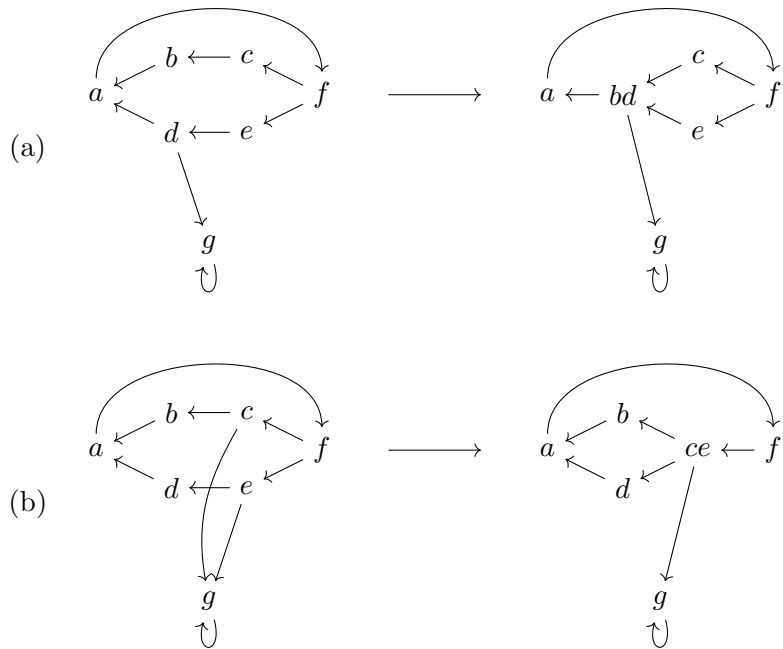

To verify a potential conjugacy between vertex shifts efficiently, we will instead apply a more direct reduction to the irreducible case.
Given a 1-block code $\Phi_\infty:X_G\to X_H$ between reducible vertex shifts, we will extend $G$ and $H$ to irreducible graphs while preserving the conjugacy or non-conjugacy of $\Phi_\infty$.
The key operation for this extension is the following procedure, which adds a new sink vertex to a sink component in such a way as to preserve conjugacy/non-conjugacy.
We will then apply this procedure to every sink component, and in reverse to every source component, until we have enough structure to connect the new sink verticies back to the new source vertices through a new vertex $*$, rendering both graphs irreducible.

Let $T$ be a sink component of $H$ and $T'=\Phi^{-1}(T)$ be the subgraph of $G$ which maps to $T$ under $\Phi_\infty$.
The procedure is as follows:
\begin{enumerate}
\item Pick an arbitrary vertex $v$ in $T$.
\item Pick an arbitrary cycle $c$ in $T$ ending at $v$ of length $|c|\leq |T|$.
\item Add the vertex $t$ along with the edges $(t,t),(v,t)$ to $H$.
  Call this new graph $\hat{H}$.
\item Select the vertices $v'\in \phiv^{-1}(v)$ which are followed by an infinite word $w'$ such that $\phiw(v'w')=vc^\infty$.
  Call this set of vertices $V'$.
\item Add the vertex $t'$ and the edges $\{(v',t'):v'\in V'\cup \{t'\}\}$ to $G$.
  Call this new graph $\hat{G}$.
\item Define $\hat{\Phi}_\infty:X_{\hat{G}}\to X_{\hat{H}}$ by
  $\hat{\phiv}(u)=
  \begin{cases}
    \phiv(u), & \text{if }u\neq t' \\ t, & \text{if }u=t'
  \end{cases}$~.
\end{enumerate}

\begin{proposition}\label{prop:add_a_vertex}
  Let $\Phi_\infty:X_G \to X_H$ be a 1-block code between reducible vertex shifts.
  Then $\hat{\Phi}_\infty:X_{\hat{G}}\to X_{\hat{H}}$ as described above is a conjugacy if and only if $\Phi_\infty$ is a conjugacy.
\end{proposition}
\begin{proof}
  Since $X_G$ is a subshift of $X_{\hat{G}}$ (and similarly for $H$) and $\hat{\Phi}_\infty$ preserves $\Phi_\infty$, it immediately follows that $\Phi_\infty$ is a conjugacy whenever $\hat{\Phi}_\infty$ is.
  For the converse, suppose $\hat{\Phi}_\infty$ is not a conjugacy.

  If $\hat{\Phi}_\infty$ is not injective,
  we have distinct points $p_1,p_2\in X_{\hat G}$ such that $\hat{\Phi}_\infty(p_1)=\hat{\Phi}_\infty(p_2)=q$.
  If $q\in X_H$, then $p_1,p_2\in X_G$ by the definition of $\hat{\Phi}_\infty$, so $\Phi_\infty$ is not injective.
  If $q\notin X_H$, then $q=wvt^\infty$.
  By the definition of $\hat{\Phi}_\infty$, we have $p_1=w_1v_1t'^\infty,p_2=w_2v_2t'^\infty$.
  By the construction of $N^-(t')$, there exist infinite words $w_1',w_2'$ such that $v_1w_1',v_2w_2'$ are words in $G$ and $\phiw(w_1')=c^\infty=\phiw(w_2')$.
  Thus $\Phi_\infty(w_1v_1w_1')=\Phi_\infty(w_2v_2w_2')$, and $\Phi_\infty$ is not injective.

  If $\hat{\Phi}_\infty$ is not surjective,
  then there exists $p\in X_{\hat{H}}$ which is not mapped to.
  If $p\in X_H$, then $\Phi$ is not surjective.
  Otherwise, $p \notin X_H$, so $p=wvt^\infty$.
  But again noting the construction of $N^-(t)$, the point $wvc^\infty$ is not mapped to.
\end{proof}

We now construct the final graphs $G^*$, $H^*$ using the above procedure as well as one additional step below.
Let $T_1,\ldots,T_m$ be the sink components of $H$, and $S_1,\ldots,S_\ell$ the source components, with $T_i'=\Phi^{-1}(T_i)$ and $S_i'=\Phi^{-1}(S_i)$ the corresponding inverse image subgraphs of $G$.
We apply the above procedure iteratively to every sink component, and every source component by reversing the edge direction in $G,H$, applying the procedure, and reversing edges back.
Let $\hat G, \hat H$ denote the graphs after applying the procedure to the $m$ sinks and $\ell$ sources.
Note that, by construction, each sink or source component in $\hat H$ contains a single state.
Furthermore, the preimage of each of these states under the induced map $\hat{\phiv}$ in $\hat{G}$ also contains a single state.
Denote the source states in $\hat{H}$ as $\{s_1,\ldots,s_\ell\}$ and the sink states as $\{t_1,\ldots,t_m\}$.
In $\hat{G}$, denote the preimages as $s_i'=\phiv^{-1}(s_i), t_j'=\phiv^{-1}(t_j)$.
We extend $\hat{H},\hat{G}$ to the irreducible graphs $H^*,G^*$ as follows:
\begin{enumerate}
\item Add a new vertex $*$ to both $\hat{H}$ and $\hat{G}$.
\item In $H^*$, set $N^-(*)=\{t_1,\ldots,t_m\}$, $N^+(*)=\{s_1,\ldots,s_\ell\}$.
  In $G^*$, set $N^-(*)=\{t_1',\ldots,t_m'\}$, $N^+(*)=\{s_1',\ldots,s_\ell'\}$.
\item Define $\Phi_\infty^*:X_{G^*}\to X_{H^*}$ by $\phiv^*(u)=
  \begin{cases}
    \hat{\phiv}(u), &\text{if }u\neq * \\ *,&\text{if } u=*
  \end{cases}$~.
\end{enumerate}

\begin{proposition}
  Let $\Phi_\infty:X_G\to X_H$ be a 1-block code between reducible vertex shifts.
  Then $\Phi_\infty^*:X_{G^*}\to X_{H^*}$ as described in the construction above is a conjugacy if and only if $\Phi_\infty$ is a conjugacy.
\end{proposition}
\begin{proof}
  By Proposition~\ref{prop:add_a_vertex}, $\hat\Phi_\infty:X_{\hat G}\to X_{\hat H}$ is a conjugacy if and only if $\Phi_\infty$ is a conjugacy, where as above the graphs $\hat G, \hat H$ immediately precede the addition of the vertex $*$.
  As in the proof of Proposition~\ref{prop:add_a_vertex}, $X_{\hat G}$ is a subshift of $X_{G^*}$ (and similarly for $\hat H$) and $\Phi^*_\infty$ preserves $\Phi_\infty$, so $\hat\Phi_\infty$ is a conjugacy if $\Phi_\infty^*$ is.
  For the other direction, suppose $\Phi^*_\infty$ is not a conjugacy.

  First suppose $\Phi_\infty^*$ is not injective.
  Since $G^*,H^*$ are irreducible, we have cycles $c,d$ in $G^*$ from Theorem~\ref{thm:conjugacy_iff_cycle_bijection} such that $\phic^*(c)=\phic^*(d)$.
  Without loss of generality, $c,d$ pass through $*$.
  Since $\phiv^*$ is bijective on $\{s_1,\ldots,s_\ell,t_1,\ldots,t_m\}$, we have $c=*s_iw_1t_j,d=*s_iw_2t_j$.
  But then $\hat\phiv$ collapses the diamond $(s_iw_1t_j,s_iw_2t_j)$, so by Lemma~\ref{lemma:diamond_lemma}, $\hat\Phi_\infty$ is not injective.

  Now suppose $\Phi_\infty^*$ is not surjective.
  Again by Theorem~\ref{thm:conjugacy_iff_cycle_bijection}, we know there is a cycle $c$ which is not in the image of $\phic^*$.
  Without loss of generality, we can assume $c=*s_iwt_j$.
  By the construction of $N^-(*),N^+(*)$, we conclude that $s_iwt_j$ is not in the image of $\hat\phiw$.
  Thus $s_i^\infty w t_j^\infty$ is a point in $X_{\hat H}$ which is not in the image of $\hat\Phi_\infty$.
\end{proof}

We now have that given reducible vertex shifts $X_G,X_H$ and a proposed 1-block conjugacy between them, the shifts can be embedded into irreducible shifts such that the conjugacy or non-conjugacy is preserved.
Next we show this embedding can be performed efficiently; the procedure described in the proof is outlined in Algorithm~\ref{alg:is_conjugacy_reducible}.

\begin{theorem}\label{thm:construct_irreducible_supershit}
  Given reducible vertex shifts $X_G,X_H$ and a 1-block code as $\phiv:V_G\to V_H$, the graphs $G^*$ and $H^*$ can be constructed in $O(|V_G|^3)$ time.
\end{theorem}
\begin{proof}
  Let $T$ be an arbitrary sink component in $H$ and $T'$ be the subgraph $\Phi^{-1}(T)$ of $G$.
  We will show the corresponding sink vertices $t,t'$ can be added in $O(|V_{T'}|^3)$ time.
  Iterating over all sink components $T\in \mathcal{T}$ and source components $S\in \mathcal{S}$ will give an overall complexity of $O(\sum_{T'\in \mathcal{T}'}|V_{T'}|^3+\sum_{S'\in\mathcal{S}'}|V_{S'}|^3)=O(|V_G|^3)$ time.
  (Adding the vertex $*$ takes linear time.)

  Let $v$ be an arbitrary vertex of $T$, and let $c$ be the shortest cycle in $T$ through $v$, which can be computed using breadth-first search in $O(|V_{T}|+|E_T|)=O(|V_H|^2)=O(|V_G|^2)$ time.
  Note that $|c|\leq|T|$, so we have completed steps~1 and~2.
  Step~3 is constant time.
  The only nontrivial step that remains is step~4, the computation of the set $V' \subseteq V_{T'}$, from which steps~5 and~6 follow trivially in linear time.
  
  Let $C=(V_C,E_C)$ be the subgraph of $T$ corresponding to $c$, and let $C'=(V_{C'},E_{C'})$ be the subgraph of $T'$ which maps onto $C$ as follows: $V_{C'} = \Phi^{-1}(V_C)$, and $E_{C'} = \{(u',v') \in E_{T'} : (\Phi(u'),\Phi(v')) \in E_C\}$.
  The subgraph $C'$ can be constructed in $O(|V_{T'}|^2)$ time.
  Note that infinite walks in $C'$ starting from any $v'\in\Phi^{-1}(v)$ are precisely the walks in $T'$ that map onto $c^\infty$, and moreover, there is an infinite walk in $C'$ starting from $v'$ if and only if there is a path in $C'$ from $v'$ to a cycle in $C'$.
  We therefore define $V'\subseteq \Phi^{-1}(v) \subseteq V_{C'}$ to be the set of nodes $v'$ such that there is a path in $C'$ from $v'$ to a cycle in $C'$.
  To compute $V'$, we can simply run breadth-first search from each vertex in $\Phi^{-1}(v)$, in $O(|\Phi^{-1}(v)|\cdot(|V_{C'}|+|E_{C'}|))=O(|V_{T'}|^3)$ time.
\end{proof}

We have now seen an efficient procedure to embed a pair of reducible graphs into a pair of irreducible graphs, such that the original pair admits a 1-block conjugacy if and only if the embedded pair does.
Moreover, the embedded irreducible graphs have at most twice the number of vertices as the original graphs.
With this procedure in hand, we can extend our verification algorithm to the reducible case.

\begin{corollary}\label{cor:verify_conjugacy_reducible}
  Given vertex shifts $X_G,X_H$ and a $k$-block code $\Phi_\infty$ as $\phiv:V_G^k\to V_H$, deciding if $\Phi_\infty$ is a conjugacy can be determined in $O(|V_G|^{4k})$ time.
\end{corollary}
\begin{proof}
  If $G,H$ are irreducible, Corollary~\ref{cor:verify_conjugacy_irreducible} applies immediately.
  For the reducible case, as in Corollary~\ref{cor:verify_conjugacy_irreducible}, by passing to the $k$th higher block shift it suffices to show the case $k=1$.
  From Theorem~\ref{thm:construct_irreducible_supershit}, we can embed $G,H$ into the irreducible shifts $G^*,H^*$ in $O(|V_G|^3)$ time.
  Furthermore, $|V_{G^*}|<2|V_G|$, so $|V_{G^*}|=O(|V_G|)$.
  Then by Corollary~\ref{cor:verify_conjugacy_irreducible}, we can verify if $\Phi^*_\infty$ (and hence $\Phi_\infty$) is a conjugacy in $O(|V_{G^*}|^4)=O(|V_G|^4)$.
\end{proof}

\section{Deciding $k$-Block Conjugacy}
\label{sec:conjugacy}

We now turn to the question of deciding $k$-block conjugacy.
Specifically, we wish to understand the complexity of the problem \KBCv, which is to decide given directed graphs $G,H$ whether the vertex shifts $X_G,X_H$ are conjugate via a $k$-block code $\Phi_\infty:X_G\to X_H$.
Note that the description size of $\Phi$ is polynomial in $|V_G|$ and $|V_H|$, and thus from Corollary~\ref{cor:verify_conjugacy_reducible} we know that a potential $k$-block conjugacy can be verified in polynomial time; hence, \KBCv is in \NP.
We will show that \KBCv is \GI-hard for all $k$, where \GI is the class of problems with a polynomial-time Turing reduction to the Graph Isomorphism problem~\cite{kobler2012graph}.
(A graph isomorphism is bijection between the vertices of two graphs which preserves the edges/non-edge relation; the Graph Isomorphism problem is to decide if two given undirected graphs are isomorphic.)

\begin{definition}
  Given directed graphs $G,H$, the \emph{k-Block Conjugacy Problem}, denoted \KBCv, is to decide if there is a $k$-block conjugacy $\Phi_\infty:X_G\to X_H$ between the vertex shifts $X_G$ and $X_H$.
\end{definition}

To begin, we give the straightforward result that the case $k=1$ is \GI-hard, essentially because 1-block conjugacies between vertex shifts for equal sized graphs must be isomorphisms.

\begin{theorem}\label{thm:1block_v_GI-hard}
	The 1-Block Conjugacy Problem, \BCv{1}, is \GI-hard. 
\end{theorem}
\begin{proof}
  Given strongly connected graphs directed $G,H$ with $|V_G|=|V_H|$, we show that the shifts $X_G,X_H$ are conjugate via 1-block code if and only if the graphs are isomorphic
  (cf.~\cite[Ex.~2.2.14]{lind1999introduction}).
  The result then follows as graph isomorphism between strongly connected directed graphs is \GI-hard, by the usual reduction from the undirected case (replace each edge with two directed edges).

  First suppose $\Psi:V_G\to V_H$ is a graph isomorphism.
  As $\Psi(v_1v_2)$ is a legal word in $X_H$ for all words of length 2, by definition of a graph isomorphism, we have that $\Psi_\infty : X_G \to X_H$ is a valid 1-block code.
  Letting $\Phi = \Psi^{-1}:V_H\to V_G$, we have $\Psi_\infty(\Phi_\infty((x_i)_{i\in \Z}))=(\Psi(\Phi(x_i)))_{i\in\Z}=(x_i)_{i\in\Z}$ for all $x\in X_H$, and $\Phi_\infty(\Psi_\infty((x_i)_{i\in \Z}))=(\Phi(\Psi(x_i)))_{i\in\Z}=(x_i)_{i\in\Z}$ for all $x\in X_G$.
  Thus, $\Phi_\infty$ is the 2-sided inverse of $\Psi_\infty$, and $\Psi_\infty$ is a 1-block conjugacy.
  
  For the other direction, suppose $\Phi_\infty:X_G\to X_H$ is a 1-block conjugacy.
  Then $\{\Phi(v):v\in V_G\}$ must be exactly the set of words of length 1 in $X_H$, i.e., the vertices of $H$.
  Since $|V_G|=|V_H|$, $\Phi:V_G\to V_H$ is a bijection.
  Also, for any edge $(v_1,v_2)\in E_G$, we have $\Phi(v_1v_2)=\Phi(v_1)\Phi(v_2)$, so $(\Phi(v_1),\Phi(v_2))\in E_H$ as $\Phi_\infty$ is a well-defined sliding block code.
  Even more, consider any pair $v_3,v_4$ of vertices in $V_G$ such that $(v_3,v_4)\notin E_G$.
  Noting that $\Phi(v_3)\Phi(v_4)$ has a unique preimage as $\Phi$ is a bijection and $\Phi_\infty$ is surjective, we have $(\Phi(v_3),\Phi(v_4))\notin E_H$.
  Thus $\Phi:V_G\to V_H$ is a bijection on vertices which preserves the edge relationship; that is, $\Phi$ is a graph isomorphism.
\end{proof}

Next, we will show \KBCv is \GI-hard for all $k$, by reduction to the 1-block case.
Specifically, given directed graphs $G,H$, we will construct graphs $G',H'$ such that there exists a 1-block conjugacy $\Phi_\infty:X_G\to X_H$ if and only if there exists a $k$-block conjugacy $\Phi'_\infty:X_{G'}\to X_{H'}$ exists.
To form $G'$, we replace every vertex $v\in V_G$ with a path $v_{\text{in}}v_1v_2\cdots v_{k-1}$ followed by the diamond with sides $v_{k-1}v_k^tv_{\text{out}}$ and $v_{k-1}v_k^bv_{\text{out}}$ (Figure~\ref{fig:vertex_gadget}a).
To form $H'$, we replace every vertex $u\in V_H$ with two parallel paths $u_{\text{in}}u_1^tu_2^t\cdots u_k^tu_{\text{out}}$ and $u_{\text{in}}u_1^bu_2^b\cdots u_k^bu_{\text{out}}$ (Figure~\ref{fig:vertex_gadget}b).

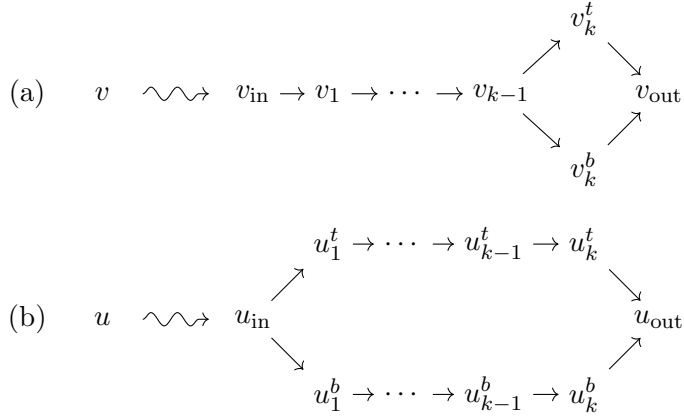
\begin{figure}
	\centering
  \begin{tikzpicture}[scale=1, decoration=snake]
  \node at (0,0) {(a)};
  \node (a) at (1,0) {$v$};
  \node (ain) at (3,0) {$v_{\text{in}}$};
  \node (a1) at (4,0) {$v_1$};
  \node (adots) at (5.05,0) {$\cdots$};
  \node (an1) at (6.3,0) {$v_{k-1}$};
  \node (ant) at (7.4,1) {$v_k^t$};
  \node (anb) at (7.4,-1) {$v_k^b$};
  \node (aout) at (8.4,0) {$v_{\text{out}}$};
  
  \draw[->] (ain) edge (a1) (a1) edge (adots) (adots) edge (an1) (an1) edge (ant) (an1) edge (anb) (ant) edge (aout) (anb) edge (aout);
  
  \node at (0,-3) {(b)};
  \node (b) at (1,-3) {$u$};
  \node (bin) at (3,-3) {$u_{\text{in}}$};
  \node (b1t) at (4,-2) {$u_1^t$};
  \node (b1b) at (4,-4) {$u_1^b$};
  \node (bdotst) at (5,-2) {$\cdots$};
  \node (bdotsb) at (5,-4) {$\cdots$};
  \node (bn1t) at (6.2,-2) {$u^t_{k-1}$};
  \node (bn1b) at (6.2,-4) {$u^b_{k-1}$};
  \node (bnt) at (7.4,-2) {$u_k^t$};
  \node (bnb) at (7.4,-4) {$u_k^b$};
  \node (bout) at (8.4,-3) {$u_{\text{out}}$};
  
  \draw[->] (bin) edge (b1t) (b1t) edge (bdotst) (bdotst) edge (bn1t) (bn1t) edge (bnt) (bnt) edge (bout) (bin) edge (b1b) (b1b) edge (bdotsb) (bdotsb) edge (bn1b) (bn1b) edge (bnb) (bnb) edge (bout);
  
  \node (swarrow) at (1.4,-3) {};
  \node (searrow) at (2.55,-3) {};
  \draw[->] (swarrow) edge[decorate] (searrow);
  
  \node (nwarrow) at (1.4,0) {};
  \node (nearrow) at (2.55,0) {};
  \draw[->] (nwarrow) edge[decorate] (nearrow);
  \end{tikzpicture}
	\caption{The vertex gadgets for (a) each vertex $v$ in $G$, and (b) each vertex $u$ in $H$.}
	\label{fig:vertex_gadget}
\end{figure}

\begin{lemma}\label{lemma:k-block_conj_format}
  Given directed graphs $G,H$, let $G',H'$ be constructed as above.
  If there exists a $k$-block conjugacy $\Phi'_\infty:X_{G'}\to X_{H'}$, then for all $v\in V_G$ there exists $u\in V_H$ such that
  $\Phi'(v_{\text{in}}v_1\cdots v_{k-1})=u_{\text{in}}$.
\end{lemma}
\begin{proof}
  Suppose for a contradiction that $\Phi'_\infty$ is a $k$-block code such that for some $v\in V_G$
  we have $\Phi'(v_{\text{in}}v_1\cdots v_{k-1})\neq u_{\text{in}}$ for all $u\in V_H$.
  We break the argument into two cases.

  First, suppose $\Phi'(v_{\text{in}}v_1\cdots v_{k-1})=u_i^t$.
  (The case $u_i^b$ is identical.)
  Since the shift map commutes with sliding block codes, we must have $\Phi'(v_1\cdots v_{k-1}v_k^t)=\Phi'(v_1\cdots v_{k-1}v_k^b) = z$ where $z\in\{u_{i+1}^t,u_{\text{out}}\}$.
  Picking any edge $(v,\hat{v})\in E_G$ and continuing to slide the block window, we must have $\Phi'({\hat{v}}_{\text{in}} {\hat{v}}_1\cdots {\hat{v}}_{k-1})\in\{{\hat{u}}_i^t,\hat{u}_i^b\}$ for some $\hat u\in V_H$.
  Without loss of generality, assume $\Phi'({\hat{v}}_{\text{in}} {\hat{v}}_1\cdots {\hat{v}}_{k-1})=\hat{u}_i^t$.
  Furthermore, since there is only one word in $H'$ between $u_i^t$ and ${\hat{u}}_i^t$ of proper length but two words in $G'$ between $\hat{v}_\text{in}$ and $v_\text{in}$, we have \[\Phi'(v_{\text{in}}\cdots v_k^tv_{\text{out}}\hat{v}_{\text{in}}\cdots \hat{v}_{k-1})=u_i^t\cdots u_k^tu_{\text{out}}\hat{u}_{\text{in}}{\hat{u}}_1^t \cdots {\hat{u}}_i^t=\Phi'(v_{\text{in}}\cdots v_k^bv_{\text{out}}\hat{v}_{\text{in}}\cdots \hat{v}_{k-1}).\]
  That is, $\Phi'$ collapses a diamond, so by Lemma~\ref{lemma:diamond_lemma}, $\Phi'$ is not a conjugacy.
	
  Second, suppose $\Phi'(v_{\text{in}}v_1\cdots v_{k-1})=u_{\text{out}}$.
  Pick any edge $(\hat{v},v)\in E_G$.
  Then without loss of generality, $\Phi'(\hat{v}_{\text{out}}v_{\text{in}}\cdots v_{k-2})=u_k^t$. 
  Continuing to slide the block window, we have $\Phi'(\hat{v}_{\text{in}}\cdots \hat{v}_{k-1})=\hat{u}_{\text{out}}$ for some $\hat u\in V_H$.
  Again, there are two words in $G'$ between $\hat{v}_{\text{in}}$ and $v_{\text{in}}$ but only one word in $H'$ between $\hat{u}_\text{out}$ and $u_\text{out}$ which passes through $u_k^t$.
  Thus, we have \[\Phi'(\hat{v}_{\text{in}}\cdots {\hat{v}}_k^t\hat{v}_{\text{out}}v_{\text{in}}\cdots v_{k-1})=\hat{u}_{\text{out}}u_{\text{in}}\cdots u_k^t u_{\text{out}}=\Phi'(\hat{v}_{\text{in}}\cdots {\hat{v}}_k^b\hat{v}_{\text{out}}v_{\text{in}}\cdots v_{k-1}),\] so $\Phi'$ again collapses a diamond, and by Lemma~\ref{lemma:diamond_lemma}, $\Phi'$ is not a conjugacy.
\end{proof}

\newcommand{\str}[2]{#1^{\langle #2\rangle}}
We now show that graphs $G,H$ admit a 1-block conjugacy if and only if the graphs $G',H'$ constructed as above admit a $k$-block conjugacy.
To do this, we first introduce a natural operation on shift spaces, which ``stretches'' each point by a factor $N$.
Given alphabet $\cA$, and any point $p=\cdots v.v'\cdots\in \cA^\Z$, we write $\str{p}{N} = \cdots v\cdots v.v'\cdots v'\cdots$ to be the point $p$ with each symbol repeated $N$ times.
Given a shift $X$ over alphabet $\cA$, we define the shift space $\str{X}{N} = \{\sigma^i(\str{p}{N}):p\in X, i\in\Z\}$ where $\sigma$ is the shift map.
In particular, $\str{X}{N}$ contains all shifts of the $N$th expansion of points in $X$.
While in general $\str{X}{N}$ is not a vertex shift when $N>1$, it is still structured enough that the following lemma is immediate.

\begin{lemma}\label{lemma:pass_to_higher_expansion}
  Given shifts $X,Y$, there exists a 1-block conjugacy $\Phi_\infty:X\to Y$ if and only if there exists a 1-block conjugacy $\str{\Phi}{N}_\infty:\str{X}{N}\to \str{Y}{N}$, where $\Phi=\str{\Phi}{N}$ as block maps.
\end{lemma}

To make use of this definition and lemma, we will project points in $X_{G'},X_{H'}$ to $\str{X_G}{k+2},\str{X_H}{k+2}$ by simply erasing the subscript and superscript information.
Formally, we define the 1-block map $\Psi^G:V_{G'}\to V_G$ by $\Psi^G(u)=v\text{ for }u\in \{v_\text{in},v_1,\ldots, v_{k-1},v_k^t, v_k^b, v_\text{out}\}$, and let $\pi^G = \Psi^G_\infty:X_{G'}\to \str{X_{G}}{k+2}$.
We define $\Psi^H,\pi^H$ similarly.
Letting $S_p^G := {\left(\pi^G\right)}^{-1}(p) \subseteq X_{G'}$, we have that $\{S_p^G : p\in \str{X_G}{k+2}\}$ is a partition of the points in $X_{G'}$.
(Similarly for $S_q^H$ and $X_{H'}$.)

\begin{theorem}\label{thm:1-block_iff_k-block}
  Given graphs $G,H$, construct $G',H'$ as above. Then there exists a 1-block conjugacy $\Phi_\infty:X_G\to X_H$ if and only if there exists a $k$-block conjugacy $\Phi':X_{G'}\to X_{H'}$.
\end{theorem}
\begin{proof}

  $(\Rightarrow)$
  Suppose there exists a 1-block conjugacy $\Phi_\infty:X_G\to X_H$. By Lemma~\ref{lemma:pass_to_higher_expansion}, there is a 1-block conjugacy $\str{\Phi}{k+2}_\infty:\str{X_G}{k+2}\to \str{X_H}{k+2}$.
  Define the $k$-block code $\Phi'_\infty:X_{G'}\to X_{H'}$ with no memory by
  \begin{itemize}\setlength{\itemsep}{0pt}
  \item $\Phi'(v_{\text{in}}\cdots)=\Phi(v)_{\text{in}}$
  \item $\Phi'(v_i\cdots v_k^t \cdots)=\Phi(v)_i^t$, $i \in \{1,\ldots,k\}$
  \item $\Phi'(v_i\cdots v_k^b \cdots)=\Phi(v)_i^b$, $i \in \{1,\ldots,k\}$
  \item $\Phi'(v_{\text{out}}\cdots)=\Phi(v)_{\text{out}}$
  \end{itemize}
  To show that $\Phi'_\infty$ is a bijection, we will show that for any $p\in \str{X_G}{k+2}$ and $q\in \str{X_H}{k+2}$ with $\str{\Phi}{k+2}_\infty(p)=q$, the map $\Phi'_\infty:S^G_p\to S^H_q$ is a bijection.
  The result then follows because $\str{\Phi}{k+2}_\infty$ is a bijection between $\str{X_G}{k+2}$ and $\str{X_H}{k+2}$, and the sets $\{S^G_p:p\in\str{X_G}{k+2}\},\{S^H_q:q\in\str{X_H}{k+2}\}$ partition $X_{G'},X_{H'}$.

  We first claim that $\Phi'_\infty(S^G_p)\subseteq S^H_q$, which is to say, for every $p'\in X_{G'}$ such that $\pi^G(p') = p$, we have $\pi^H(\Phi'_\infty(p')) = q$.
  To see this, note that by construction of $\Phi'$, for all $p'\in X_{G'}$ and all $i\in\Z$, we have $\Psi^H(\Phi'_\infty(p')_i) = \Phi(\Psi^G(p'_i)) = \str{\Phi}{k+2}(\Psi^G(p'_i))$.
  The condition $\pi^G(p') = \Psi^G_\infty(p') = p$ implies $\Psi^G(p'_i) = p_i$ for all $i\in\Z$.
  Combining the above with the observation that $\str{\Phi}{k+2}(p_i) = q_i$ gives $\Psi^H(\Phi'_\infty(p')_i) = q_i$, which implies the claim.
  
  To see that $\Phi'_\infty$ is injective on $S^G_p$, consider distinct points $p',p''\in S^G_p$ which differ at index $i$.
  Since $\pi^G(p')=\pi^G(p'')$, we can assume without loss of generality that $p'_i=v_k^t$ and $p''_i=v_k^b$.
  Then \[\Phi'(p'_{[i-k,i]})=\Phi'(v_1\cdots v_{k-1}v_k^t)=\Phi(v)_1^t\neq\Phi(v)_1^b=\Phi'(v_1\cdots v_{k-1}v_k^b)=\Phi'(p''_{[i-k,i]}),\] so $\Phi'_\infty(p')\neq \Phi'_\infty(p'')$.
  
  $(\Leftarrow)$
  Suppose there exists a $k$-block conjugacy $\Phi'_\infty: X_{G'}\to X_{H'}$.
  Without loss of generality, assume $\Phi'_\infty$ has no memory.
  By Lemma~\ref{lemma:k-block_conj_format}, for every $v\in V_G$, there exists $u\in V_H$ such that $\Phi'(v_{\text{in}}\cdots)=u_{\text{in}}$.
  Define the 1-block code $\Phi_\infty:X_G\to X_H$ by $\Phi(v)=\Psi^H(\Phi'(v_{\text{in}}v_1\cdots v_{k-1}))$.
  We claim $\Phi_\infty$ is a conjugacy.
  To show this, we instead will show $\str{\Phi_\infty}{k+2}$ defined by the same block map is a conjugacy.
  To see $\str{\Phi_\infty}{k+2}$ is surjective, consider any $q\in \str{X_H}{k+2}$.
  Picking any $q'\in S^H_{q}$, set $p'={\Phi'}_\infty^{-1}(q')$ and $p=\pi^G(p')$.
  We will now show $\str{\Phi_\infty}{k+2}(p)=q$, so the diagram in Figure~\ref{fig:commutativediagram_p-q-p'-q'} is commutative and $\Phi_\infty$ is surjective.
  For all $i\in\Z$,
  \begin{figure}
    \centering
    \begin{tabular}{ccc}
    \begin{tikzcd}
    X_{G'}\arrow{r}{\Phi'_\infty}\ar{d}[swap]{\pi^G=\Psi^G_\infty} & X_{H'}\ar{d}{\pi^H=\Psi^H_\infty} \\
    \str{X_G}{k+2} \arrow{r}{\str{\Phi_\infty}{k+2}} & \str{X_H}{k+2} 
    \end{tikzcd}
    
    & &
    
    \begin{tikzcd}
    p' \arrow[mapsto]{r}{\Phi'_\infty} \arrow[mapsto]{d}[swap]{\pi^G} & q' \arrow[mapsto]{d}{\pi^H} \\
    p \arrow[mapsto]{r}{\str{\Phi_\infty}{k+2}} & q
    \end{tikzcd}
    \end{tabular} 
    \caption{Given $\Phi'$ or $\Phi$, one can construct the other such that this diagram commutes.}
    \label{fig:commutativediagram_p-q-p'-q'}
  \end{figure}
  \begin{equation}\label{eq:1-block-code-definition}
    \Phi(p_i)=\Psi^H(\Phi'((p_i)_\text{in}\cdots (p_i)_{k-1}))=\Psi^H((q_i)_\text{in})=q_i.
  \end{equation}
  
  To see $\str{\Phi_\infty}{k+2}$ is injective, consider distinct $p_1,p_2\in \str{X_G}{k+2}$.
  Then $S^G_{p_1}\neq S^G_{p_2}$.
  Since $\Phi'_\infty$ is a conjugacy, $\Phi'_\infty(S^G_{p_1})\cap\Phi'_\infty(S^G_{p_2})=\emptyset$.
  By Lemma~\ref{lemma:k-block_conj_format}, there exist $q_1,q_2\in \str{X_H}{k+2}$ such that $S^H_{q_1}=\Phi'_\infty(S^G_{p_1})$ and $S^H_{q_2}=\Phi'_\infty(S^G_{p_2})$.
  By the construction of $\Phi$ (and shown in~(\ref{eq:1-block-code-definition})), $\str{\Phi_\infty}{k+2}(p_1)=q_1\neq q_2=\str{\Phi_\infty}{k+2}(p_2)$, so $\str{\Phi_\infty}{k+2}$ is injective.
\end{proof}

The construction in Theorem~\ref{thm:1-block_iff_k-block} gives a polynomial-time reduction from \BCv{1} to \KBCv for all $k$.
(The same construction could plausibly give a reduction from \BCv{m} to \BCv{\ell} where $\ell = (m-1)(k+2)+k$, though if true the proof would be much more involved.)
Combining this reduction with Theorem~\ref{thm:1block_v_GI-hard} therefore gives \GI-hardness for all $k$.

\begin{corollary}\label{cor:KBC_GI-hard}
  \KBCv is \GI-hard for all $k$.
\end{corollary}

\section{Reducing Representation Size}
\label{sec:reducing}

Thus far we have addressed two problems.
We first gave an efficient algorithm, given directed graphs $G,H$ and $k$-block map $\Phi$, to verify whether $\Phi_\infty:X_G\to X_H$ is a conjugacy.
We then showed that the problem of deciding whether $X_G$ and $X_H$ are conjugate, given only $G$ and $H$, is \GI-hard.
We now address a problem given only $G$ and an integer $\ell$: whether we can find a $k$-block code which reduces the size of $G$ by $\ell$ vertices while preserving conjugacy.

\begin{definition}\label{def:1-block_reduction}
  Given a directed graph $G$ and integer $\ell$, the \emph{$k$-Block Reduction Problem}, denoted \BRv{k}, is to decide if there exists a directed graph $H$ with $|V_H|=|V_G|-\ell$ such that the vertex shifts $X_G$ and $X_H$ are conjugate via a $k$-block code.
\end{definition}

We will show this problem is \NP-complete for the case $k=1$, by modifying the hardness proof of the State Amalgamation Problem (\SAP), which asks if $\ell$ consecutive amalgamations can be performed on a graph $G$~\cite{SAP}.
The proof that \SAP is \NP-hard shows that the set of graphs satisfying a certain structure property is closed under the amalgamation operation.
This structure is then leveraged to encode an \NP-complete problem (Hitting Set).
While 1-block codes are more general than sequences of amalgamations (Figure~\ref{fig:a_dot_map}), we find that, surprisingly, the same set of graphs is also closed under 1-block conjugacy.
In fact, the rest of the construction of~\cite{SAP} suffices as well, though much of the argument needs to be strengthened to the general 1-block case.

We begin by recalling the structure property.

\begin{definition}[\cite{SAP}]
  A directed graph $G$ satisfies the \emph{structure property} if it is essential and there exists a partition $\{\{\alpha\}, A,B,C\}$ of $V_G$ such that the following four conditions hold.
  \begin{enumerate}
    \item $N^+(\alpha)=N^-(\alpha)=\{\alpha\} \cup A\cup C$.
    \item For each $a\in A$, $N^-(a)=\{a, \alpha\}$ and $\{a,\alpha\}\subseteq N^+(a)\subseteq \{a,\alpha\}\cup B$.
    \item For each $c\in C$, $N^+(c)=\{c,\alpha\}$ and $\{c,\alpha\}\subseteq N^-(c)\subseteq\{c,\alpha\}\cup B$.
    \item For each $b\in B$, $N^-(b)\subseteq A$ and $N^+(b)\subseteq C$.
  \end{enumerate}
\end{definition}
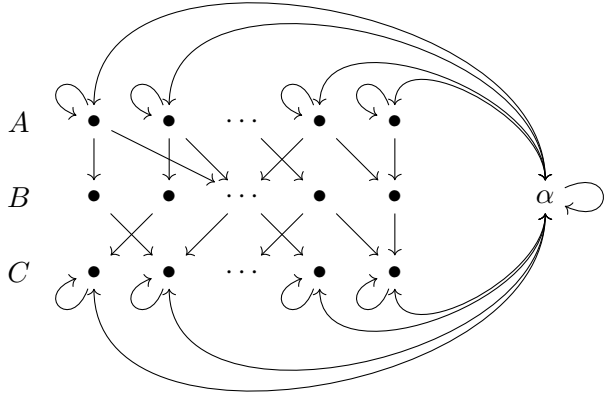
\begin{figure}
  \centering
  \begin{tikzpicture}
  \node (b1) at (0,0) {$\bullet$};
  \node (b2) at (1,0) {$\bullet$};
  \node (b3) at (2,0) {$\cdots$};
  \node (b4) at (3,0) {$\bullet$};
  \node (b5) at (4,0) {$\bullet$};
  \node (b) at (-1,0) {$C$};
  
  \node (i1) at (0,1) {$\bullet$};
  \node (i2) at (1,1) {$\bullet$};
  \node (i3) at (2,1) {$\cdots$};
  \node (i4) at (3,1) {$\bullet$};
  \node (i5) at (4,1) {$\bullet$};
  \node (i) at (-1,1) {$B$};
  
  \node (a1) at (0,2) {$\bullet$};
  \node (a2) at (1,2) {$\bullet$};
  \node (a3) at (2,2) {$\cdots$};
  \node (a4) at (3,2) {$\bullet$};
  \node (a5) at (4,2) {$\bullet$};
  \node (a) at (-1,2) {$A$};
  
  \node (alpha) at (6,1) {$\alpha$};
  
  \draw[<->] (a1) edge[out=90, in=90] (alpha) (a2) edge[out=90, in=90] (alpha) (a4) edge[out=90, in=90] (alpha) (a5) edge[out=90, in=90] (alpha);
  
  \draw[<->] (b1) edge[out=270, in=270] (alpha) (b2) edge[out=270, in=270] (alpha) (b5) edge[out=270, in=270] (alpha) (b4) edge[out=270, in=270] (alpha);
  
  \draw[->] (b1) edge[out=250, in=200, looseness=8] (b1) (b2) edge[out=250, in=200, looseness=8] (b2) (b4) edge[out=250, in=200, looseness=8] (b4) (b5) edge[out=250, in=200, looseness=8] (b5);
  
  \draw[->] (a1) edge[out=110, in=160, looseness=8] (a1) (a2) edge[out=110, in=160, looseness=8] (a2) (a4) edge[out=110, in=160, looseness=8] (a4) (a5) edge[out=110, in=160, looseness=8] (a5);
  
  \draw[->] (alpha) edge[out=25,in=-25,looseness=8] (alpha);
  
  \draw[->] (a1) edge (i1) (a1) edge (i3) (a2) edge (i2) (a2) edge (i3) (a4) edge (i3) (a4) edge (i5) (a5) edge (i5) (i1) edge (b2) (i2) edge (b1) (i3) edge (b4) (i4) edge (b3) (i4) edge (b5) (i5) edge (b5) (i3) edge (b2) (a3) edge (i4);
  \end{tikzpicture}
  \caption{A graph which satisfies the structure property}
  \label{fig:structure_property}
\end{figure}

See Figure~\ref{fig:structure_property} for an example.
We now show that the structure property is preserved under 1-block conjugacy.

\begin{lemma}\label{lemma:structure_lemma}
  Let $G$ be a graph with the structure property having $\{\{\alpha\}, A,B,C\}$ as the partition of $V_G$, and let $\Phi_\infty:X_G\to X_H$ be a 1-block conjugacy.
  Then $\phiv(v_1)=\phiv(v_2)$ implies $v_1=v_2$ or $v_1,v_2\in B$, so $H$ also satisfies the structure property with vertex partition $\{\{\phiv(\alpha)\},\phiv(A),\phiv(B),\phiv(C)\}$. 
\end{lemma}
\begin{proof}
  First note that if $\Phi:X_G\to X_H$ is a 1-block conjugacy from a graph $G$ with vertex partition $\{\{\alpha\},A,B,C\}$ such that $\Phi(v_1)=\Phi(v_2)$ implies $v_1=v_2$ or $v_1,v_2\in B$, then the fact that $H$ satisfies the structure property with vertex partition $\{\{\phiv(\alpha)\},\phiv(A),\phiv(B),\phiv(C)\}$ follows immediately.
  Now suppose for a contradiction that $v_1\neq v_2\in V_G$ and $\phiv(v_1)=\phiv(v_2)$; we proceed in cases.

  \noindent
  Case 1: $v_1,v_2\in\{\alpha\}\cup A\cup C$.
  Then $\Phi_\infty(v_1^\infty)=\Phi_\infty(v_2^\infty)$ and $\Phi_\infty$ is not a conjugacy.

  \noindent
  Case 2: $v_1=\alpha,v_2\in B$.
  Let $a\in A, c\in C$ be such that $av_2c$ is a word in $X_G$.
  (Such $a,c$ exist as $G$ is essential.)
  Then $\Phi_\infty((av_2c\alpha)^\infty)=\Phi_\infty((a\alpha c\alpha)^\infty)$ and $\Phi_\infty$ is not a conjugacy.

  \noindent
  Case 3a: $v_1\in A,v_2\in B,(v_1,v_2)\notin E_G$.
  Let $a\in A$ be such that $(a,v_2)\in E_G$.
  Note that $a\neq v_1$.
  Consider the point \[p=(\phiv(a)\phiv(v_2)\phiv(\alpha))^\infty=(\phiv(a)\phiv(v_1)\phiv(\alpha))^\infty\] in $X_H$ of period 3.
  Due to $G$ having the structure property and our assumption that $\Phi_\infty$ is a 1-block conjugacy, the preimage of $p$ must be defined by a 3-cycle whose vertices are contained in $\{\alpha\}\cup A\cup C$.
  In particular, the preimage must trace a self-loop, so we know $\phiv(a)=\phiv(\alpha)$ or $\phiv(a)=\phiv(v_1)$ or $\phiv(v_1)=\phiv(\alpha)$.
  Since we know $\phiv$ is injective on $\{\alpha\}\cup A\cup C$ by Case 1, none of these are possible.

  \noindent
  Case 3b: $v_1\in A, v_2\in B,(v_1,v_2)\in E_G$.
  Let $c\in C$ be such that $(v_2,c)\in E_G$.
  Consider the point \[p=(\phiv(v_2)\phiv(c)\phiv(\alpha))^\infty=(\phiv(v_1)\phiv(c)\phiv(\alpha))^\infty\] in $X_H$ of period 3.
  Again by the requirement that the preimage of $p$ traces a self-loop, we know $\phiv(v_1)=\phiv(c)$ or $\phiv(v_1)=\phiv(\alpha)$ or $\phiv(\alpha)=\phiv(c)$.
  However, all of these situations violate the injectivity of $\phiv$ on $\{\alpha\}\cup A\cup C$.

  \noindent
  Case 4: $v_1\in C,v_2\in B$.
  This is identical to Case 3 where the edges in the graph have been reversed.
\end{proof}

As in~\cite{SAP}, we will need a ``weight widget'' which acts as a weighted switch, using the following notation.
Let $v$ be a vertex with $N^-(v)=D$ and $N^+(v)=E$.
We will write $v:[D,E]$ in this situation, and as a slight abuse of notation, we will drop the curly brackets if $E$ or $D$ is a singleton and write $v:[u,E]$.
Additionally, we extend this notation to sets $S$ of vertices in the obvious way and write $S:[D,E]$.

\begin{definition}[\cite{SAP}]
  Let $G$ satisfy the structure property with $V_G=\{\alpha\}\cup A\cup B\cup C$, and let $K>0$ be a fixed even integer.
  Then for nonempty subsets $A_*\subseteq A, C_*\subseteq C$, the \emph{weight widget} $w=\weight[A_*,C_*]$ is the following collection of vertices.
  \begin{itemize}\setlength{\itemsep}{0pt}
    \item $A_w=\{a_1^w,\ldots,a_{K/2}^w\}$
    \item $B_w=\{b_1^w\ldots,b_K^w\}$
    \item $C_w=\{c_1^w,\ldots,c_{K/2}^w\}$
  \end{itemize}
  where $A_w\cap A_*=\emptyset = C_w\cap C_*$, and for each $i\in \{1,\ldots,K/2\}$ we have
  \begin{itemize}\setlength{\itemsep}{0pt}
    \item $b_{2i-1}:[A_*\cup \{a_1^w,\ldots,a^w_{i-1}\},c^w_i]$
    \item $b_{2i}:[a_i^w,C_*\cup \{c_1^w,\ldots,c^w_i\}]$.
  \end{itemize}
  Moreover, we require these to be the only images of $A_w$ in $B$, i.e., $B\cap N^+(a_i^w)\subseteq B_w$ for all $a_i^w\in A_w$, and similarly for the preimage of $C_w$.
  For a given 1-block conjugacy $\Phi_\infty$, letting $S=\Phi^{-1}(\Phi(b_1^w))\setminus B_w = \{b\in B: \Phi(b)=\Phi(b^w_1)\}\setminus B_w$, we say $w$ is \emph{activated} if $S:[A_*,C_*]$.
\end{definition}
See Figure~\ref{fig:weight_widget} for an example.
The term ``activate'' comes from the following fact, which we show below in Lemma~\ref{lemma:weight_lemma}(1): if $S$ is a singleton, then the construction of the weight widget allows the states $S\cup B_w$ to be amalgamated sequentially into a single state.
For example, the vertex $v$ in Figure~\ref{fig:weight_widget} can activate the weight widget shown.
The next two lemmas show that these amalgamations cannot be performed if the widget is not activated.

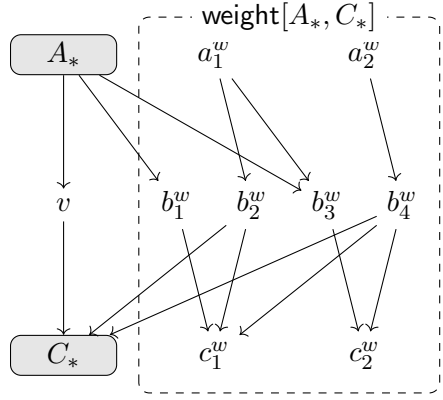
\begin{figure}
  \centering
  \begin{tikzpicture}
  \tikzstyle{set} = [rectangle, rounded corners, draw=black, fill=black!10, text width=1.2cm, text centered, inner sep=3pt]
  \node[set] (C) at (0.5,-0.5) {$C_*$};
  \node (v) at (0.5,1.5) {$v$};
  \node[set] (A) at (0.5,3.5) {$A_*$};
  \node (b1) at (2,1.5) {$b_1^w$};
  \node (b2) at (3,1.5) {$b_2^w$};
  \node (b3) at (4,1.5) {$b_3^w$};
  \node (b4) at (5,1.5) {$b_4^w$};
  \node (a1) at (2.5,3.5) {$a_1^w$};
  \node (a2) at (4.5,3.5) {$a_2^w$};
  \node (c1) at (2.5,-0.5) {$c_1^w$};
  \node (c2) at (4.5,-0.5) {$c_2^w$};
  
  \draw[dashed, rounded corners] (1.5,-1) rectangle (5.5,4);
  
  \node[fill=white] (weight) at (3.5,4) {$\weight[A_*,C_*]$};
  
  \draw[->] (a1) edge (b2) (a1) edge (b3) (a2) edge (b4) (b1) edge (c1) (b2) edge (c1) (b3) edge (c2) (b4) edge (c2) (b4) edge (c1);
  
  \draw[->] (A) edge (b1) (A) edge (b3) (A) edge (v) (v) edge (C) (b2) edge (C) (b4) edge (C);
  \end{tikzpicture}
  \caption{The weight widget $\weight[A_*,B_*]$ with $K=4$.}
  \label{fig:weight_widget}
\end{figure}

\begin{lemma}\label{lemma:no_w_i_left_behind}
  Let $w=\weight[A_*,C_*]$ be a weight widget in $G$.
  If $\Phi_\infty:X_G\to X_H$ is a 1-block conjugacy between graphs with the structure property, then for any $v\in V_H$, the statement $b^w_\ell\in \Phi^{-1}(v)$ for $\ell>1$ implies $b^w_{\ell-1}\in \Phi^{-1}(v)$ or $|\Phi^{-1}(v)|=1$.
\end{lemma}
\begin{proof}
  By contrapositive, suppose $|\Phi^{-1}(v)|>1$ and there exists $b^w_\ell\in \Phi^{-1}(v)$ such that $b^w_{\ell-1}\notin \Phi^{-1}(v)$.
  Without loss of generality, let $\ell$ be the largest such subscript.
  We have two cases.
  
  \noindent
  Case 1: $\ell$ is even.
  We claim there must exist $a\in N^-(v)\setminus\{\Phi(a^w_{\ell/2}),\ldots,\Phi(a^w_{K/2})\}$.
  To see this claim, first note the weight widget construction guarantees that for every $b\in B\setminus B_w$, we have $N^-(b)\cap A_w=\emptyset$.
  Also by the weight widget construction, for any $b^w_{2i+1}\in B_w$, we have $A_*\subseteq N^-(b^w_{2i+1})\setminus A_w$.
  Noting that $\weight[A_*,B_*]$ is not defined when $A_*=\emptyset$ and $G$ is not essential when $N^-(b)=\emptyset$, if we have $\Phi^{-1}(v)\setminus\{b^w_{2i}:b^w_{2i}\in B_w\}\neq\emptyset$, then we must have $N^-(v)\setminus \Phi(A_w)\neq\emptyset$.
  That is, the claim is satisfied in the case when $\Phi^{-1}(v)\not\subseteq\{b^w_{2i}:b^w_{2i}\in B_w\}$.
  Assuming now that $\Phi^{-1}(v)\subseteq\{b^w_{2i}:b^w_{2i}\in B_w\}$, we note our earlier assumption that $|\Phi^{-1}(v)|>1$ guarantees there exists $b^w_{2j}\in \Phi^{-1}(v)$ for some index $2j\neq \ell$.
  By our other assumption that $\ell$ is the largest such subscript, we have $2j< \ell$.
  Then $\Phi(a^w_{j})\in N^-(v)$, and the claim follows.
  Proceeding, we then have $\Phi(a)v\Phi(c^w_{\ell/2})$ is a word in $X_H$; however, $N^-(c^w_{\ell/2})= \{b^w_{\ell-1}\}\cup \{b^w_\ell,b^w_{\ell+2},\ldots,b^w_K\}$ and for all $b_{2i}^w\in N^-(c^w_{\ell/2})\setminus\{b^w_{\ell-1}\},(a,b_{2i}^w)\notin E_G$.
  Since $b_{\ell-1}^w\notin \Phi^{-1}(v)$ by assumption, this word has no preimage in $X_G$, so $\Phi_\infty$ is not a conjugacy.
  
  \noindent
  Case 2: $\ell$ is odd.
  Using an argument symmetric to the one in case 1, we get that there must exist $c\in N^+(v)\setminus \{\Phi(c^w_{(\ell+1)/2}),\ldots,\Phi(c^w_{K/2})\}$.
  Then $\Phi(a^w_{(\ell-1)/2})v\Phi(c)$ is a word in $X_H$ with no preimage, so $\Phi_\infty$ is not a conjugacy.
\end{proof}

\begin{lemma}\label{lemma:weight_lemma}
  Suppose $w=\weight[A_*,C_*]$ is a weight widget in $G$. Then
  \begin{enumerate}[1.]
    \item Suppose $\Phi_\infty:X_G\to X_{G'}$ is a 1-block conjugacy such that $\Phi^{-1}(\Phi(b_i^w))=\{b_i^w\}$ for all $b_i^w\in B_w$ and $V_{G'}$ contains $v:[A_*,B_*]$.
    Defining $\Phi'_\infty:X_G\to X_H$ by \[\Phi'(u)=
    \begin{cases}
      \Phi(u), & \text{if }u\notin \Phi^{-1}(v)\cup B_w \\ v, & \text{if }u\in \Phi^{-1}(v)\cup B_w
    \end{cases} \] where $H$ is the minimal graph induced by $G$ and $\Phi'$, then $\Phi'$ is a 1-block conjugacy with $|V_H|=|V_{G'}|-K$.
    \item If $w=\weight[A_*,C_*]$ is not activated and $\Phi_\infty:X_G\to X_H$ is a 1-block conjugacy, then $\Phi^{-1}(\Phi(b^w_i))$ is a singleton for every $b^w_i$ with $i>1$.
  \end{enumerate}
\end{lemma}
\begin{proof}
  (1) Consider the sequence of splittings followed by the sequence of amalgamations which transforms $G$ into $G'$.
  Then note that by the construction of the weight widget, $\{v:[A_*,C_*]\}\cup B_w$ can be amalgamated sequentially for an additional $K$ amalgamations.
  
  (2) Suppose $b^w_1\in\Phi^{-1}(v)$ for some $v\in V_H$ and consider $V=\Phi^{-1}(v)\setminus B_w$.
  By Lemma~\ref{lemma:no_w_i_left_behind} it suffices to show $b^w_2\notin \Phi^{-1}(v).$
  By definition of $w$ not being activated, we have two cases.
  
  \noindent
  Case 1: $N^-(V)\neq A_*$.
  If there is some $a\in N^-(V)\setminus A_*$, then $\Phi(a)v\Phi(c^w_1)$ is a word in $X_H$.
  Since there is no state in $G$ connecting $a$ with $c^w_1$, the word has no preimage in $X_G$ and $\Phi_\infty$ is not a conjugacy.
  Otherwise, there is some $a\in A_*\setminus N^-(V)$.
  By contrapositive, suppose $b^w_2\in \Phi^{-1}(v)$.
  Picking any $c\in C_*$, we have $\Phi(a)v\Phi(c)$ is a word in $X_H$.
  Since there is no state in $\Phi^{-1}(v)$ connecting $a$ with $c$, the word has no preimage and $\Phi_\infty$ is not a conjugacy.
  
  \noindent
  Case 2: $N^+(V)\neq C_*$.
  By contrapositive, suppose $b^w_2\in \Phi^{-1}(v)$.
  If there is some $c\in N^+(V)\setminus C_*$, then $\Phi(a^w_1)v\Phi(c)$ is a word in $X_H$.
  Since there is no state in $G$ connecting $a^w_1$ with $c$, the word has no preimage and $\Phi_\infty$ is not a conjugacy.
  Otherwise, there is some $c\in C_*\setminus N^+(V)$.
  Considering any $a\in N^-(V)$, we have $\Phi(a)v\Phi(c)$ is a word in $X_H$.
  Since there is no state in $\Phi^{-1}(v)$ connecting $a$ with $c$, the word has no preimage and $\Phi_\infty$ is not a conjugacy.
\end{proof}

We now define the Hitting Set problem, which is \NP-complete~\cite{karp1972reducibility}, and state a lemma which we will need in the proof.

\begin{definition}
  Let $\cS=\{S_1,\ldots,S_m\}$ be a collection of sets with $\bigcup_i S_i=U$. Given a subset $S\subseteq U$, we define its \emph{hit set} as $\hit(S)=\{S_i:S\cap S_i\neq \emptyset\}$. Given $\cS,U$, and an integer $t$, the \emph{hitting set problem}, denoted \Hit, is to decide whether there is a set $H$ of cardinality $t$ such that $\hit(H)=\cS$.
  We will also overload this notation, and write $\hit(s)$ to mean $\hit(\{s\})$ for $s\in U$.
\end{definition}

\begin{lemma}[\cite{SAP}]\label{lemma:hitH-H} 
  Let $(\cS,U,t)$ be an instance of \Hit. Suppose for some $t\leq |\cS|$ there is no $H$ with $|H|\leq t$ and $\hit(H)=\cS$. Then for all $H\subseteq U$, $|\hit(H)|-|H|<|\cS|-t$.
\end{lemma}

\begin{figure}
  \centering
  \begin{tikzpicture}
  \node (b1) at (0,0) {$b_{S_1u_1}$};
  \node (b2) at (1.5,0) {$b_{S_1u_2}$};
  \node (b3) at (3,0) {$b_{S_2u_2}$};
  \node (b4) at (4.5,0) {$b_{S_2u_3}$};
  \node (b5) at (6,0) {$b_{S_1\beta}$};
  \node (b6) at (7.5,0) {$b_{S_2\beta}$};
  
  \node (a1) at (2.5,2) {$S_1$};
  \node (a2) at (5,2) {$S_2$};
  
  \node (c1) at (0.6,-2) {$u_1$};
  \node (c2) at (2.7,-2) {$u_2$};
  \node (c3) at (4.8,-2) {$u_3$};
  \node (c4) at (6.9,-2) {$\beta$};
  
  \draw[->] (a1) edge (b1) (a1) edge (b2) (a2) edge (b3) (a2) edge (b4) (a1) edge (b5) (a2) edge (b6);
  \draw[->] (b1) edge (c1) (b2) edge (c2) (b3) edge (c2) (b4) edge (c3) (b5) edge (c4) (b6) edge (c4);
  
  \node (alpha) at (10,0) {$\alpha$};
  
  \draw[<->] (a1) edge[bend left, in=120, out=45] (alpha) (a2) edge[bend left, in=130, out=45] (alpha);
  
  \draw[<->] (c1) edge[bend right, in=-110, out=-45] (alpha) (c2) edge[bend right, in=-115, out=-45] (alpha) (c3) edge[bend right, in=-124, out=-45] (alpha) (c4) edge[bend right, in=-140, out=-45] (alpha);
  
  \draw[->] (a1) edge[loop left] (a1) (a2) edge[loop left] (a2) (c1) edge[loop left] (c1) (c2) edge[loop left] (c2) (c3) edge[loop left] (c3) (c4) edge[loop left] (c4) (alpha) edge[loop right] (alpha);
  \end{tikzpicture}
  \caption{The graph constructed in Theorem~\ref{thm:1block-NPcomplete} for the \Hit instance with $\cS=\{\{u_1,u_2\},\{u_2,u_3\}\}$, without any weight widgets attached.}
  \label{fig:hitting_set_reduction}
\end{figure}
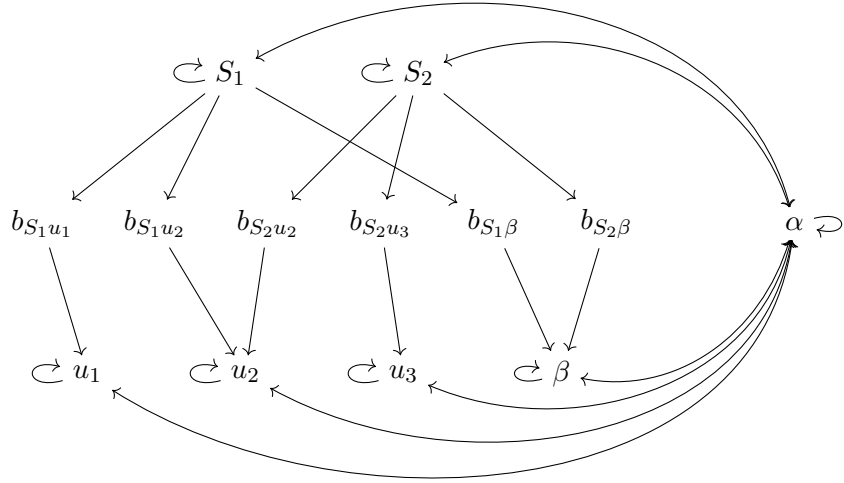

We now show that \BRv{1} is \NP-complete, by reduction from \Hit.
Given the lemmas developed above, the result essentially follows from the argument in~\cite{SAP}, with minor modifications for the 1-block case; for completeness, we give the full proof.

\begin{theorem}\label{thm:1block-NPcomplete}
  \OBRv is \NP-complete.
\end{theorem}
\begin{proof}
  First we show \OBRv is in \NP.
  Given a vertex shift $X_G$ and $\phiv:V_G\to\{1,2,\ldots, |V_G|-n\}$ from a proposed 1-block conjugacy $\Phi_\infty$, we construct the minimal image graph $G'$ such that $\Phi_\infty:X_G\to X_{G'}$ is well-defined.
  In particular, $V_{G'}=\{\phiv^{-1}(u):u\in V_G\}$ and $E_{G'}=\{(\phiv^{-1}(v_1),\phiv^{-1}(v_2)):(v_1,v_2)\in E_G\}$.
  By Corollary~\ref{cor:verify_conjugacy_reducible}, we can determine if $\Phi_\infty$ is a conjugacy in $O(|V_G|^4)$ time.
  
  To show hardness, we reduce from \Hit; let $\cS=\{S_1,\ldots, S_m\}$ be the collection of sets and $t$ the given integer.
  Defining $n=|U|$ for $U=\bigcup_i S_i$, we set the parameter $K=5mn$ for the weight widgets.
  Then, as in~\cite{SAP}, we build the following graph $G=(V_G,E_G)$ with the structure property $V_G=A\cup B\cup C\cup \{\alpha\}$.
  \begin{enumerate}
    \item Start with $A=\cS,B=\emptyset, C=U\cup \{\beta\}$, where $\beta$ is a new vertex.
    \item For each $u\in A,v\in C$, add $b_{uv}:[u,v]$. That is, add the vertex $b_{S_is}$ and path $S_i\to b_{S_is}\to s$ for every $(s,S_i)$ with $s\in S_i$ as well as the vertex $b_{S_i\beta}$ and path $S_i\to b_{S_i\beta}\to \beta$ for every $S_i\in\cS$.
    \item For each $(s,S_i)$ with $s\in S_i$, add the weight widget $w=\weight[S_i,\{s,\beta\}]=(A_w,B_w,C_w)$. Note that $A_w$ will added to $A$, $B_w$ to $B$, and $C_w$ to $C$.
    \item For each $s\in U$, add the weight widget $\weight[\hit(s),\{s\}]$.
    \item Finally, add the vertex $\alpha$ and the necessary edges for $G$ to have the structure property, i.e., add the edges $\{(a,\alpha),(\alpha,a):a\in A\}\cup\{(b,\alpha),(\alpha,b):b\in B\}\cup \{(v,v):v\in A\cup B\cup\{\alpha\}\}$.
  \end{enumerate}
  Summarizing, if $W$ is the collection of weight widgets added, $A=\cS\cup \bigcup_{w\in W}A_w$, $B=\{b_{S_is}:S_i\in \cS, s\in S_i\}\cup \{b_{S_i\beta}:S_i\in \cS\}\cup \bigcup_{w\in W}B_w$, and $C=U\cup \{\beta\}\cup \bigcup_{w\in W}C_w$.
  (See Figure~\ref{fig:hitting_set_reduction} for an example with $\cS=\{\{u_1,u_2\},\{u_2,u_3\}\}$ where only steps (1), (2), and (5) have been performed.)
  
  We will show there is a hitting set of size $t$ if and only if there is a 1-block conjugacy $\Phi_\infty:X_G\to X_{G'}$ such that $|V_{G'}|\leq|V_G|-(m+n-t)K$.
  The idea behind the reduction is that $s$ can either choose to be in the hitting set by combining some $b_{S_is}$ with the appropriate $b_{S_i\beta}$ to activate some of the $\weight[S_i,\{s,\beta\}]$, or choose not to be in the hitting set by combining all $b_{S_is}$ for $S_i\in \hit(s)$ to activate $\weight[\hit(s),\{s\}]$.
  We will be able to activate $|\hit(H)|+|U\setminus H|=m+(n-t)$ weight widgets if there is a hitting set of size $t$ and strictly fewer if no such set exists.
  By our choice of $K$, any reduction in the number of vertices not caused by activating weight widgets will be insignificant.
  
  First, suppose there is a hitting set $H$ for $\cS$ of size $t$.
  We will give a sequence of $(m+n-t)K$ consecutive amalgamations, which together constitute a 1-block reducing the number of vertices by $(m+n-t)K$.
  For each $S_i\in \cS$, pick some $s\in H$ such that $S_i\in \hit(s)$. After amalgamating $b_{S_is}$ with $b_{S_i\beta}$, the weight widget $w=\weight[S_i,\{s,\beta\}]$ can be activated and $B_w$ amalgamated sequentially.
  Doing this for each $S_i$ gives a total of $m(K+1)\geq mK$ consecutive amalgamations.
  As the above amalgamations only affected the vertices in $B$ associated with $H$, next consider any $s\in U\setminus H$.
  We can amalgamate the vertices $\{b_{S_is}:S_i\in \hit(s)\}$ in any order to form $b_{\hit(s)s}:[\hit(s),\{s\}]$ which can then activate $\weight[\hit(s),\{s\}]$.
  Amalgamating all the vertices in this weight widgets give a total of at least $K$ amalgamations for each vertex $s\in U\setminus H$.
  Thus we can perform $mK+(n-t)K=(m+n-t)K$ consecutive amalgamations, so there is a 1-block conjugacy $\Phi_\infty:X_G\to X_{G'}$ such that $|V_G|\geq |V_{G'}|-(m+n-t)K$.
  
  Next suppose there is no hitting set $H$ of size $t$.
  Let $\Phi:X_G\to X_{G'}$ be a 1-block conjugacy such that $N=|V_G|-|V_{G'}|$ is as large as possible.
  Define
  \begin{align*}
    \overline{H}&=\{s\in U:\weight[\hit(s),\{s\}]\text{ is activated} \}, \\
    F&=\{S_i:\weight[S_i,\{s,\beta\}]\text{ is activated for some }s\in S_i \}, \\
    H&=U\setminus \overline{H}.
  \end{align*}
  Note that there is a single path in $G$ from $S_i$ to $s$, through the vertex $b_{S_is}$, which is required to activate both $\weight[\hit(s),\{s\}]$ and $\weight[S_i,\{s,\beta\}]$.
  Thus for every $b_{S_is}$ we have that if $S_i\in F$, then $s\in H$. That is,
  \begin{align}\label{eq:F_subset_ellipsis}F\subseteq \{S_i:s\in H \text{ for some }b_{S_is} \}. \end{align}
  We now count how much smaller $|V_{G'}|$ could be than $|V_G|$.
  By construction, each activated widget can lead to reducing the number of vertices by at most $K$.
  Let $B_{\text{non-weight}}=\{b_{S_is}:S_i\in\cS,s\in S_i \}\cup \{b_{S_i\beta}:S_i\in \cS \}$ be the vertices in $B$ not in weight widgets.
  By Lemma~\ref{lemma:weight_lemma}, if $u\in\Phi^{-1}(v)$ with $|\Phi^{-1}(v)|>1$ for some $u$ not in an activated widget, then $u\in B_\text{non-weight}\cup \bigcup_{w\in W}w_1$.
  Thus $V_G$ can be reduced by at most $(|F|+|\overline{H}|)K+|B_\text{non-weight}|+|W|.$
  Since \[|B_\text{non-weight}|+|W|=(mn+m)+(mn+n)<K, \] we have
  \begin{align*}
    N &\leq (|F|+|\overline{H}|)K+|B_\text{non-weight}|+|W| \\
    &< (|F|+|\overline{H}|)K+K \\
    &\leq (|\hit(H)|+(n-|H|)+1)K &\text{(by (\ref{eq:F_subset_ellipsis}) and $H=U\setminus\overline{H}$)}\\
    &\leq (m+n-t)K &\text{(by Lemma \ref{lemma:hitH-H})}.
  \end{align*}
\end{proof}

\section{Edge Shifts}
\label{sec:edge-shifts}

Thus far we have restricted our attention to vertex shifts, rather than edge shifts, though the latter are perhaps more commonly used in the literature.
For various reasons, the problems we consider are in general more appropriate for vertex shifts, as we discuss in the following section.
(Vertex shifts are also motivated by applications (\S~\ref{sec:discussion}).)
Nonetheless, we now give some results for edge shifts, for the first two problems: verifying $k$-block conjugacies, and testing pairs of shifts for conjugacy.
(The third problem remains open.)

In the following, we will leverage our results for vertex shifts, using the standard conversion from edge shifts to vertex shifts: edges become vertices, and pairs of adjacent edges become edges~\cite[Proposition 2.3.9]{lind1999introduction}.
More formally, we recall that given edge shift $X_G^e$, its vertex shift representation is the shift $X_{G'}$ where $V_{G'}=E_{G'}$ and $E_{G'}=\{(e_i,e_j):e_ie_j\text{ is a word in $X_G^e$}\}$.
Thus, for any edge shifts $X_G^e,X_H^e$, there exists a $k$-block conjugacy $\Phi_\infty:X_G^e\to X_H^e$ if and only if there exists a $k$-block conjugacy $\Phi'_\infty:X_{G'}\to X_{H'}$ between the vertex shift representations of $X_G$ and $X_H$.

First, we observe that our verification algorithm for vertex shifts immediately applies to edge shifts.

\begin{theorem}
  \label{thm:edge-verification}
  Given directed multigraphs $G,H$ and a proposed $k$-block code $\Phi_\infty:X_G^e\to X_H^e$, deciding if $\Phi_\infty$ is a conjugacy can be determined in $O(|E_G|^{4k})$.
\end{theorem}
\begin{proof}
  Given edge shifts $X_G^e,X_H^e$, we first construct their vertex shift representations $X_{G'},X_{H'}$ as above.
  Letting $\Phi'_\infty$ be the corresponding block code between the vertex shifts, by Corollary~\ref{cor:verify_conjugacy_reducible}, we can determine if $\Phi'_\infty$ is a conjugacy in $O(|V_{G'}|^{4k})=O(|E_G|^{4k})$ time.
\end{proof}

We now turn to the $k$-block conjugacy problem for edge shifts, where we again show \GI-hardness.

\begin{definition}
  Given directed mult-graphs $G,H$, the \emph{$k$-Block Conjugacy Problem}, denoted \KBCe, is to decide is there is a $k$-block conjugacy $\Phi_\infty: X_G^e \to X_H^e$ between the edge shifts $X_G^e,X_H^e$.
\end{definition}

\begin{figure}
  \centering
  \begin{tikzpicture}[scale=1, decoration=snake]
  \node at (-1,0) {(a)};
  \node (e1) at (0,0) {$v$};
  \node (e2) at (1,0) {$v'$};
  \node (ain) at (3,0) {$v$};
  \node (a1) at (4,0) {$\bullet$};
  \node (adots) at (5.05,0) {$\cdots$};
  \node (an1) at (6.3,0) {$\bullet$};
  \node (an) at (7.4,0) {$\bullet$};
  \node (anplus1) at (8.4,0) {$\bullet$};
  \node (aout) at (9.4,0) {$v'$};
  
  \draw[->] (e1) edge node[above]{$e$} (e2);
  
  \draw[->] (ain) edge node[above]{$e_{\text{in}}$} (a1) (a1) edge node[above]{$e_1$} (adots) (adots) edge node[above]{$e_{k-2}$} (an1) (an1) edge node[above]{$e_{k-1}$} (an) (an) edge[bend left] node[above]{$e_{k}^t$} (anplus1) (an) edge[bend right] node[below]{$e_{k}^b$} (anplus1) (anplus1) edge node[above]{$e_\text{out}$} (aout);
  
  \node at (-1,-3) {(b)};
  \node (f1) at (0,-3) {$u$};
  \node (f2) at (1,-3) {$u'$};
  \node (bin) at (3,-3) {$u$};
  \node (b1) at (4,-3) {$\bullet$};
  \node (bdotst) at (5,-2) {$\bullet$};
  \node (bdotsb) at (5,-4) {$\bullet$};
  \node (bn1t) at (6.2,-2) {$\cdots$};
  \node (bn1b) at (6.2,-4) {$\cdots$};
  \node (bnt) at (7.4,-2) {$\bullet$};
  \node (bnb) at (7.4,-4) {$\bullet$};
  \node (bout) at (8.4,-3) {$\bullet$};
  \node (boutplus1) at (9.4,-3) {$u'$};
  
  \draw[->] (f1) edge node[above]{$f$} (f2);
  \draw[->] (bin) edge node[above]{$f_\text{in}$} (b1) (b1) edge node[above left]{$f_1^t$} (bdotst) (bdotst) edge node[above]{$f_2^t$} (bn1t) (bn1t) edge node[above]{$f_{k-1}^t$} (bnt) (bnt) edge node[above right]{$f_{k}^t$} (bout) (b1) edge node[below left]{$f_1^b$} (bdotsb) (bdotsb) edge node[above]{$f_2^b$} (bn1b) (bn1b) edge node[above]{$f_{k-1}^b$} (bnb) (bnb) edge node[below right]{$f_{k}^b$} (bout) (bout) edge node[above]{$f_\text{out}$} (boutplus1);
  
  \node (swarrow) at (1.4,-3) {};
  \node (searrow) at (2.55,-3) {};
  \draw[->] (swarrow) edge[decorate] (searrow);
  
  \node (nwarrow) at (1.4,0) {};
  \node (nearrow) at (2.55,0) {};
  \draw[->] (nwarrow) edge[decorate] (nearrow);
  \end{tikzpicture}
  
  \caption{(a) The edge gadget for each pre-image graph. (b) The edge gadget for each image graph.}
  \label{fig:edge_gadget}
\end{figure}
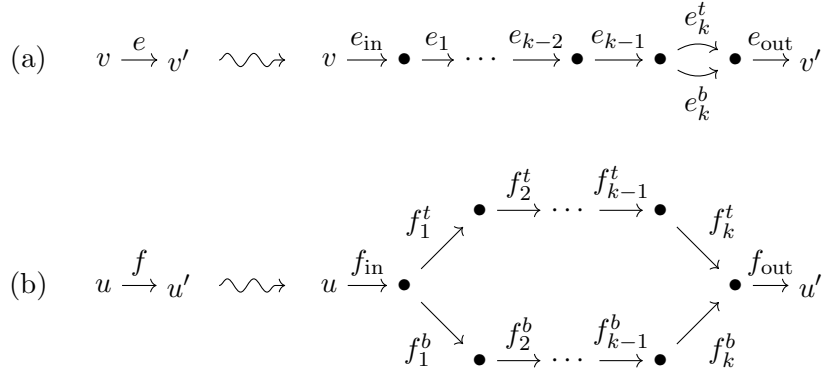

\begin{theorem}
  \KBCe is \GI-hard.
\end{theorem}
\begin{proof}
  We first show that \BCe{1} is \GI-hard.
  Given directed graphs $G,H$ with $|E_G|=|E_H|$, as in the vertex shift case, we will argue that there exists a 1-block conjugacy between the edge shifts if and only if the graphs are isomorphic.
  Suppose first that $G,H$ are isomorphic.
  Let $G',H'$ be the directed graphs for vertex shifts, as described above, so that $X_{G'}=X_G^e$ and $X_{H'}=X_H^e$.
  Since $G,H$ are isomorphic and $G',H'$ are created by the same (deterministic) procedure, $G',H'$ are isomorphic.
  By Theorem~\ref{thm:1block_v_GI-hard}, there exists a 1-block conjugacy $\Phi_\infty:X_{G'}\to X_{H'}$, so $X_{G'}=X_G^e$ is conjugate to $X_{H'}=X_H^e$ via a 1-block code.
  
  Now suppose $\Phi_\infty:X_G^e\to X_H^e$ is a 1-block conjugacy. 
  Noting that $\Phi:E_G\to E_H$ is a map on edges, we show that $\Phi$ can be realized as a map on $V_G$.
  To do this, it suffices to show (i) for any two edges $(v_1,v_2),(v_1,v_3)$ starting at the same vertex, $\Phi((v_1,v_2)),\Phi((v_1,v_3))$ also start at the same vertex and (ii) for any two edges $(u_1,u_2),(u_3,u_2)$ ending at the same vertex, $\Phi((u_1,u_2)),\Phi((u_3,u_2))$ also end at the same vertex.
  To see condition (i), consider any $(v_4,v_1)\in E_G$.
  As $\Phi((v_4,v_1)(v_1,v_2)),\Phi((v_4,v_1)(v_1,v_3))$ must both be words in $X_H^e$, we must have that $\Phi((v_1,v_2)),\Phi((v_1,v_3))$ both start at the same vertex.
  Similarly, for condition (ii), consider any $(u_2,u_4)\in E_G$, and note that $\Phi((u_1,u_2)(u_2,u_4)),\Phi((u_3,u_2)(u_2,u_4))$ are both words in $X_H^e$, so $\Phi(u_1,u_2),\Phi(u_3,u_2)$ must end at the same vertex.
  Thus $\Phi$ can be realized as a map $\Psi:V_G\to V_H$ on vertices which is surjective and preserves the edge/non-edge relation.
  To show $\Psi$ is actually a graph isomorphism, consider the inverse $\Phi^{-1}_\infty$.
  Since $\Phi_\infty$ is 1-block conjugacy and $|E_G|=|E_H|$, $\Phi^{-1}_\infty$ is a 1-block code.
  Again, $\Phi^{-1}_\infty$ can be realized as a surjective vertex map $\Psi'$ which preserves the edge/non-edge relation.
  Since both $\Psi:V_G\to V_H$ and $\Psi':V_H\to V_G$ are surjective maps between finite sets, we actually have $\Psi,\Psi'$ are bijections.
  Thus $\Psi$ is a graph isomorphism from $G$ to $H$.

  We now reduce \KBCe to \BCe{1}, as we did with vertex shifts.
  Given edge shifts $X_G^e,X_H^e$, construct $\hat{G},\hat{H}$ as follows.
  To form $\hat{G}$, substitute each edge in $G$ with a path of length $k$ followed by two parallel edges and a final edge (Figure~\ref{fig:edge_gadget}a).
  Construct $\hat{H}$ by substituting each edge in $H$ with a single edge followed by two parallel paths of length $k$ followed by a single edge (Figure~\ref{fig:edge_gadget}b).
  Then construct the vertex shift representations $G',H',\hat{G}',\hat{H}'$ of $G,H,\hat{G},\hat{H}$.
  By construction of the edge gadget, $\hat{G}',\hat{H}'$ can be formed from $G',H'$ by using the vertex gadget in Figure~\ref{fig:vertex_gadget}.
  Thus by Theorem~\ref{thm:1-block_iff_k-block}, there exists a 1-block conjugacy $\Phi'_\infty:X_{G'}\to X_{H'}$ if and only if there exists a $k$-block conjugacy $\hat{\Phi}'_\infty:X_{\hat{G}'}\to X_{\hat{H}'}$.
  Since there exists a $k$-block conjugacy $\Phi_\infty:X_{G}^e\to X_{H}^e$ between edge shifts if and only if there exists a $k$-block conjugacy $\Phi'_\infty:X_{G'}\to X_{H'}$ between the vertex representations, there exists a 1-block conjugacy $\Phi_\infty: X_G\to X_H$ if and only if there exists a $k$-block conjugacy $\hat{\Phi}:X_{\hat{G}}\to X_{\hat{H}}$.
  Thus \KBCe is \OBCe-hard and, in particular, \GI-hard.
\end{proof}

\begin{table}
	\centering
 
	\begin{tabular}{@{}lcccccccc@{}}\toprule
      & Block size & Verification ($G,H,\Phi$) & Conjugacy ($G,H$) & Reduction ($G,\ell$)\\ 
      \cmidrule{2-5}
      \multirow{2}{*}{Vertex}
        & $k=1$ & \BVv{1}: \P & \BCv{1}: \GI-hard, \NP &  \BRv{1}:  \NP-complete\phantom{??} \\
        & $k>1$ & \BVv{k}: \P & \BCv{k}: \GI-hard, \NP &  \BRv{k}:  \NP-complete?? \\[8pt]
      \multirow{2}{*}{Edge}                                               
        & $k=1$ & \BVe{1}: \P & \BCe{1}: \GI-hard, \textsf{NP}* & \BRe{1}:   \NP-complete?? \\
        & $k>1$ & \BVe{k}: \P & \BCe{k}: \GI-hard, \textsf{NP}* & \BRe{k}:   \NP-complete?? \\
  
		\bottomrule
	\end{tabular}
	\caption{Summary of results and open questions, for vertex and edge shifts.  Question marks denote conjectures, and $\mathsf{BV}$ refers to the verification problem (\S~\ref{sec:verification}).  The asterisk (*) denotes a subtlety in edge shift representations: the $k$-block conjugacy problem is in \NP when the the representation size is considered to be the number of edges (i.e., a unary representation), but membership in \NP is not clear when the shift is given as an adjacency matrix (i.e., a binary representation).}
	\label{table:complexity_table}
\end{table}

\section{Discussion}
\label{sec:discussion}

We have addressed several variants of the conjugacy problem restricted to $k$-block codes, with new algorithms to verify a proposed conjugacy, and hardness results for $k$-block conjugacy and representation reduction via 1-block codes (Table~\ref{table:complexity_table}).
Below we discuss subteties of input representation, followed by applications and open problems.

\paragraph{Representations of SFTs.}
When considering how to describe a subshift of finite type (SFT), three representations come to mind: a vertex shift, an edge shift, and a list of forbidden words $\cF$.
As our results pertain to vertex and edge shifts, we now discuss some nuances in these two representations, leaving lists of forbidden words to future work.

Perhaps the central advantage of edge shifts over vertex shifts is their compact representation size: a shift on $n$ symbols can be represented in size as small as $O(\log n)$ by writing the multi-graph as a integer adjacency matrix, as opposed to $\Omega(n)$ for vertex shifts.
This compact representation size can have important implications on the computational complexity.
In the verification problem, for example, writing down a $k$-block code $\Phi$ na\"ively takes $\Omega(n) = \Omega(|E_G|)$ space, which can be exponential in the size of the graphs $G,H$.
(One can improve this by encoding $\Phi$ as a integer $|V_G|\times|V_G|\times|E_H|$ tensor, specifying how many $(u,v)\in E_G$ edges map to a given $e\in E_H$, but this can still be exponential.)
Thus, while our algorithm remains polynomial-time, it would not be for cases allowing a compact representation of $\Phi$.

Similarly, for the conjugacy problem, we only know \KBCe to be in \NP if we consider the graphs $G,H$ to be represented in adjacency list form, which takes $\Omega(|E_G|)$ space, rather than the typically more compact integer adjacency matrix form taking $O(|V_G|\log|E_G|)$ space, as the natural certificate is the block map $\Phi$ witnessing the conjugacy.
For the matrix representation of edge shifts, membership in \NP would require a certificate exponentially smaller than the na\"ive representation of the block map $\Phi$.

Finally, what ``size reduction'' means for edge shifts depends on the choice of adjacency list or matrix above.
For the adjacency list, we have that the problem of reducing the number of vertices in the graph is in \NP, but it is less motivated, as the size is dominated by $|E_G|$, not $|V_G|$.
On the other hand, while the adjacency matrix representation size is dominated by $|V_G|$, it is not clear whether the problem of reducing the number of vertices is in \NP, for the same reason as above.

\paragraph{Motivation from Markov partitions.}
As noted in~\cite{SAP}, variants of the conjugacy problem for vertex shifts have applications in simplifying Markov partitions, a tool to study discrete-time dynamical systems via symbolic dynamics.
Briefly, a Markov partition is a collection $C$ of regions of the phase space, satisfying certain properties, which induces a conjugacy to a vertex shift $X_G$ where $G=(C,E)$, i.e., the vertices are labeled with the regions of the phase space.
In applications, one can encounter Markov partitions with thousands of regions, thus motivating the problem of simplifying the partition.
Without additional information about the dynamical system, essentially the only way to do this while preserving the relevant geometric information is to \emph{coarsen} the partition, by replacing sets of regions with a single region which is their union.
This operation is exactly a 1-block code.
Our results therefore give an efficient algorithm to test whether a proposed coarsening (1-block code) is valid (yields a conjugacy).
Our results also imply that the problem of minimizing the partition size is \NP-complete.
(Previous work~\cite{SAP} only showed the latter for the case where the 1-block code was a sequence of amalgamations.)

\paragraph{Open problems.}
Our work leaves several open problems, such as those implied by Table~\ref{table:complexity_table}:
resolving the complexity of the $k$-block conjugacy problem, and showing \NP-hardness of the size reduction problem.
The complexity of deciding $k$-block conjugacy between edge shifts represented as integer matrices is especially interesting, as membership in \NP is perhaps unlikely (see above).
Regarding the $k$-block conjugacy problem and resolving where it lies on the spectrum between \GI-complete and \NP-complete, we conjecture that, similar to the induced subgraph isomorphism problem~\cite{SYSLO198291}, it is \GI-complete when $|V_G|=|V_H|$ and \NP-complete when $|V_G|-|V_H|$ is large enough.
Beyond these questions, it would be interesting to address the complexity of $k$-block conjugacy between SFTs given as lists of forbidden words, and the natural variants of the problem for that input (for example, reducing the representation size of the list).

\subsection*{Acknowledgments}
We thank Luke Meszar for valuable contributions to the early stages of this work.
We also thank Mike Boyle, Josh Grochow, Doug Lind, and Brian Marcus, for several helpful conversations, references, and insights.

\appendix

\section{Algorithms}

\vspace{0mm} 

\begin{algorithm}[H]\label{alg:phi_c_injective} 
  \SetAlgoLined
  \SetKwFunction{GetStronglyConnectedComponents}{GetStronglyConnectedComponents}
  \SetKwFunction{IsInjective}{IsInjective}
  \SetKw{Continue}{continue}
  \SetKwProg{Function}{Function}{:}{}
  \Function{\IsInjective{$G,H,\Phi$}}{
  \KwIn{irreducible graphs $G,H$ and a 1-block code $\Phi$}
  \KwOut{true, if $\phic:\bigcup_n C_n(G)\to \bigcup_n C_n(H)$ is injective; false, otherwise}
  \BlankLine
  \tcc{Construct the meta-graph $M$}
  $V_M\leftarrow V_G\times V_G$\;
  $E_M\leftarrow \{((v_1,v_2),(u_1,u_2)):\Phi(v_1)=\Phi(u_1),\Phi(v_2)=\Phi(u_2),\text{ and }(v_1,u_1),(v_2,u_2)\in E_G\}$\;
  \BlankLine
  \tcc{Decide if $M$ has a cycle passing through $(v_1,v_2)$ with $v_1\neq v_2$}
  $\mathcal{S} \leftarrow$ \GetStronglyConnectedComponents{$M$} \tcc*{Tarjan's}
  \ForEach{subgraph s in $\mathcal{S}$}
  {
  	\If{s is a singleton}{\Continue\;}
    \ForEach{vertex $(v_1,v_2)$ in s}
    {
      \If{$v_1\neq v_2$}{\Return{true}\;}
    }
  }
  \Return{false}\;
}
  \caption{Determine if $\phic$ is injective}
\end{algorithm}

\vspace{0mm}

\begin{algorithm}[H]\label{alg:phi_c_bijective}
  \SetKwProg{Function}{Function}{:}{}
  \SetKwFunction{IsConjugacyIrreducible}{IsConjugacyIrreducible}
  \Function{\IsConjugacyIrreducible{$G,H,\Phi$}}{
  \KwIn{irreducible graphs $G,H$ and a 1-block code $\Phi$}
  \KwOut{true, if $\Phi_\infty$ is a conjugacy; false, otherwise}
  \If{not \IsInjective{$G,H,\Phi$}}{\Return{false}\;}
  \For{$i\in\{1,\ldots,|V_G|\}$}{
    \If{$\tr(A(G)^i)\neq \tr(A(H)^i)$}{\Return{false}\;}
  }
  \Return{true}\;
}
  \caption{Determine if $\Phi_\infty$ between irreducible graphs is a conjugacy}
\end{algorithm}

\vspace{0mm}

\begin{algorithm}[H]\label{alg:add_sink_components}
  \SetKwProg{Function}{Function}{:}{}
  \SetKwFunction{AddSinkComponents}{AddSinkComponents}
  \Function{\AddSinkComponents{$G,H,\Phi$}}{
    \SetAlgoLined
    \SetKwFunction{AddSourceComponents}{AddSourceComponents}
    \SetKwFunction{GetSinkVertices}{GetSinkVertices}
    \SetKwFunction{GetSourceVertices}{GetSourceVertices}
    \SetKwFunction{GetSinkComponents}{GetSinkComponents}
    \SetKwFunction{GetShortestCycle}{GetShortestCycleStartingAt}
    \SetKwFunction{GetRandomVertex}{GetRandomVertex}
    \SetKwFunction{length}{Length}
    \SetKwFunction{contains}{Contains}
    \SetKwFunction{get}{Get}
    \SetKwFunction{add}{Add}
    \SetKw{Continue}{continue}
    \KwIn{reducible graphs $G,H$ and a 1-block code $\Phi$}
    \KwResult{(1) alters $G,H$ so each sink component $T$ in $H$ has the property $|V_{\Phi^{-1}(T)}|=1$, and (2) extends $\Phi$ to the new graphs so $\Phi_\infty:X_G\to X_H$ is a conjugacy if and only if the original 1-block code was a conjugacy}
    $\mathcal{T}\leftarrow$ \GetSinkComponents{H}\;
    \ForEach{subgraph $T$ in $\mathcal{T}$}
    {
      $T'\leftarrow \Phi^{-1}(T)$\;
      \If{$|V_{T'}|=1$}{\Continue\;}
      \tcc{Find the subgraphs $C$ and $C'$}
      $v\leftarrow\GetRandomVertex{T}$\;
      $c\leftarrow\GetShortestCycle{v}$\;
      $V_C\leftarrow\{u\in V_T: u\in c \}$\;
      $E_C\leftarrow\{(u,u'):uu'\text{ is a word of length 2 contained in }c^\infty \}$\;
      $V_{C'}\leftarrow \{u\in V_{T'}:\Phi(u)\in V_C\}$\;
      $E_{C'}\leftarrow\{(u,u')\in E_{T'}:(\Phi(u),\Phi(u'))\in E_C \}$\;
      \BlankLine
      \tcc{Attach the new vertices $t$ and $t'$}
      $V_G$.\add{$t'$}\;
      $N^+(t')\leftarrow \{t'\}$\;
      $N^-(t')\leftarrow \{t'\}$\;
      \ForEach{vertex $u$ in $V_{C'}$}
      {
        \If{$\Phi(u)=v\land\text{there is a path in $C'$ from $u$ to a cycle}$}
        {
          $N^-(t').\add{u}$\;
        }
      }
      $V_H.\add{t}$\;
      $N^+(t)\leftarrow \{t\}$\;
      $N^-(t)\leftarrow \{t, v\}$\;
      $\Phi(t')\leftarrow t$\;
    }
  }
  \caption{Turn every sink component into a single vertex}
\end{algorithm}

\vspace{0mm}

\begin{algorithm}[H]\label{alg:add_source_components_wrapper_function}
  \SetKwProg{Function}{Function}{:}{}
  \SetKwFunction{reverseEdges}{ReverseEdges()}
  \Function{\AddSourceComponents{$G,H,\Phi$}}{
    \SetAlgoLined	
    \KwIn{reducible graphs $G,H$ and a 1-block code $\Phi$}
    \KwResult{(1) alters $G,H$ so each source component $S$ in $H$ has the property $|V_{\Phi^{-1}(S)}|=1$, and (2) extends $\Phi$ to the new graphs so $\Phi_\infty:X_G\to X_H$ is a conjugacy if and only if the original 1-block code was a conjugacy}
    $G$.\reverseEdges\;
    $H$.\reverseEdges\;
    \AddSinkComponents{$G,H,\Phi$}\;
    $G$.\reverseEdges\;
    $H$.\reverseEdges\;
  }
  \caption{Turn every source component into a single vertex}
\end{algorithm}

\vspace{0mm}

\begin{algorithm}[H]\label{alg:is_conjugacy_reducible}
  \SetKwProg{Function}{Function}{:}{}
  \SetKwFunction{IsConjugacyReducible}{IsConjugacyReducible}
  \Function{\IsConjugacyReducible{$G,H,\Phi$}}{

  \KwIn{reducible graphs $G,H$ and a 1-block code $\Phi$}
  \KwOut{true, if $\Phi_\infty$ is a conjugacy; false, otherwise}
  \AddSinkComponents{$G,H,\Phi$}\;
  \AddSourceComponents{$G,H,\Phi$}\;
  $V_G.\add{$v_G$}$\;
  $N^-(v_G)\leftarrow \GetSinkVertices{G}$\;
  $N^+(v_G)\leftarrow \GetSourceVertices{G}$\;
  $V_H.\add{$v_H$}$\;
  $N^-(v_H)\leftarrow \GetSinkVertices{H}$\;
  $N^+(v_H)\leftarrow \GetSourceVertices{H}$\;
  $\Phi(v_G)\leftarrow v_H$\;
  \Return{\IsConjugacyIrreducible{$G,H,\Phi$}}\;
}
  \caption{Determine if $\Phi_\infty$ between reducible graphs is a conjugacy}
\end{algorithm}

\end{document}